\documentclass[a4paper]{amsart}

\usepackage[all]{xy}
\usepackage{amsmath}
\usepackage{amssymb}
\usepackage{latexsym}
\usepackage{amsthm}
\usepackage{mathrsfs}
\usepackage{float}
\usepackage{geometry}
\usepackage{mathdots}
\usepackage{url}

\newcommand{\A}{\mathbb{A}}

\newcommand{\C}{\mathbb{C}}

\newcommand{\cD}{\mathcal{D}}
\newcommand{\disc}{\mathfrak{d}}
\newcommand{\E}{\mathbb{E}}
\newcommand{\F}{\mathbb{F}}

\newcommand{\M}{\operatorname{M}}

\newcommand{\cO}{\mathcal{O}}

\newcommand{\Q}{\mathbb{Q}}

\newcommand{\rQ}{\operatorname{Q}}
\newcommand{\R}{\mathbb{R}}

\renewcommand{\u}{\mathfrak{u}}

\newcommand{\rUra}{\operatorname{U}_{\operatorname{r}}^{\operatorname{(a)}}}
\newcommand{\V}{\mathcal{V}}

\newcommand{\W}{\mathcal{W}}
\newcommand{\Z}{\mathbb{Z}}

\newcommand{\ch}{\operatorname{ch}}

\newcommand{\GL}{\operatorname{GL}}

\newcommand{\Ind}{\operatorname{Ind}}

\newcommand{\ord}{\operatorname{ord}}

\newcommand{\diag}{\operatorname{diag}}

\newcommand{\Sp}{\operatorname{Sp}}
\newcommand{\SO}{\operatorname{O}}
\newcommand{\rSOa}{\operatorname{O}^{\operatorname{(a)}}}
\newcommand{\SL}{\operatorname{SL}}
\newcommand{\std}{\operatorname{std}}
\newcommand{\End}{\operatorname{End}}

\newcommand{\Res}{\operatorname{Res}}

\newcommand{\tru}{\bigtriangleup}
\newcommand{\trd}{\bigtriangledown}
\newcommand{\rU}{\operatorname{U}}
\newcommand{\rqGL}{\operatorname{qGL}}
\newcommand{\rQGL}{\operatorname{QGL}}

\newtheorem{df}{Definition}[section]
\newtheorem{thm}[df]{Theorem}
\newtheorem{prop}[df]{Proposition}
\newtheorem{lem}[df]{Lemma}
\newtheorem{cor}[df]{Corollary}

\newtheorem{rem}[df]{Remark}

\makeatletter
    
    \@addtoreset{equation}{section}
\makeatother

\title{On the local doubling $\gamma$-factor for classical groups 
       over function fields}

\author{Hirotaka KAKUHAMA}

\email{hkaku@math.kyoto-u.ac.jp}
\address{Department of Mathematics, Kyoto University, Kitashirakawa Oiwake-cho, Sakyo-ku, Kyoto 606-8502, Japan}

\date{}

\begin{document}

\maketitle


\begin{abstract}
In this paper, we give a precise definition of an analytic $\gamma$-factor of 
an irreducible representation of a classical group over a local function field 
of odd characteristic so that it satisfies some notable properties which are 
enough to define it uniquely. We use the doubling method to define the $\gamma$-factor, 
and the main theorem extends works of Lapid-Rallis, Gan, Yamana, and the author to a classical group over  a local function field of odd characteristic.
\end{abstract}

\setcounter{tocdepth}{1}
\tableofcontents



\section{
            Introduction
           }


Let $F$ be a local field and let $G$ be a classical group (i.e. a symplectic group, a special orthogonal group, a unitary group, or a quaternionic unitary group) over $F$.
In this paper, we apply the doubling method of Piatetski-Shapiro, Rallis (\cite{GPSR87, LR05, PSR86}) to irreducible representations of $G(F) \times F^\times$ ($G(F) \times E^\times$ if $G$ is a unitary group equipped with a quadratic extension $E/F$) 
in the case $\ch(F) \not=0,2$, and give a precise definition of a $\gamma$-factor $\gamma^\V(s,\pi\boxtimes\omega,\psi)$ for an irreducible representation of $G(F)$, a character $\omega$ of $F^\times$, and a nontrivial additive character $\psi$ of $F$. 
We note here that it has been established in the case $\ch(F) = 0$ \cite{Gan12, Kak19, LR05, Yam14}.
Let $G^\circ$ denote the Zariski connected component of $G$. It is expected that
\begin{align}\label{desiintro}
\gamma^\V(s, \pi\boxtimes\omega, \psi) = \gamma(s, \std\circ\phi_{\pi^\circ\boxtimes\omega}, \psi)
\end{align}
where $\pi^\circ$ is an irreducible component of the restriction of $\pi$ to $G^\circ$ (see Remark \ref{Lpar orth} for more detail), $\phi_{\pi^\circ\boxtimes\omega}$ is the $L$-parameter of $\pi^\circ\boxtimes\omega$, $\std$ is the standard homomorphism of the $L$-group ${}^L\!(G^\circ\times \GL_1)$ (${}^L\!(G^\circ\times\Res_{E/F}\GL_1)$ if $G$ is a unitary group) into $\GL_N(\C)$, and the right hand side is the $\gamma$-factor of \cite{GR10}.
We check some properties of $\gamma^\V(s, \pi\boxtimes\omega,\psi)$, which are based on \eqref{desiintro}. Moreover, we prove that some fundamental properties are enough to define it uniquely. This allows us to say that the definition is ``precise''.


Let $F$ be a field, let $E$ be a division $F$-algebra with $[E:F] = 1,2,4$, and let $\epsilon\in \{\pm1\}$. Then, a pair $\V = (V,h)$ is said to be an $\epsilon$-Hermitian space if 
\begin{itemize}
\item $V$ is a vector space over over $E$ 
        where $E$ is a division $F$ algebra with $[E:F] = 1,2,4$ 
         (in the later, we also consider split algebras (Sect.~\ref{setting}), 
\item $h$ is a map $V \times V \rightarrow E$ satisfying
        \[
        h(y,x) = \epsilon \cdot h(x,y)^*,\ \ h(x, ya + zb) = h(x,y)a + h(x,z)b
        \]
        for $x,y,z \in V$ and $a,b \in E$. Here $*$ denotes the main involution of $E$ over $F$. 
\end{itemize}
We assume that either $h$ is zero or non-degenerate. 
Put $G(\V)$ the algebraic group
\[
  \{ g \in \GL(V)| h(gx, gy) = h(x,y) \mbox{ for all $x,y \in V$}\}
\]
We also write $G(\V)$ for the group of its $F$-rational points $G(\V)(F)$ if there is no confusion. 


Now we explain our main result. We assume that $F$ is a local field of $\ch(F) \not=0,2$. For simplicity, in the introduction, we exclude the odd orthogonal cases and unitary cases.
For an irreducible representation $\pi$ of $G(\V)$, a character $\omega$ of $F^\times$, and a non-trivial additive character $\psi$ of $F$, we define the $\gamma$-factor as in \cite{Kak19, LR05}:
\[
\gamma^\V(s+\frac{1}{2}, \pi\boxtimes\omega,\psi)
=\Gamma^\V(s,\pi,\omega,A,\psi)c_\pi(-1)R(s,\pi,A,\psi)
\]
where
\begin{itemize}
\item $A$ is some element of the Lie algebra of $G(\V^\Box)$ (for definition, see Sect.~\ref{dm});
\item $\Gamma^\V(s,\pi,\omega,A,\psi)$ is a ``normalized $\Gamma$-factor'',
        which is defined in Sect.~\ref{dwf}. 
        This factor is obtained from a functional equation of doubling zeta integrals;
\item $c_{\pi}$ is the central character of $\pi$;
\item $R(s,\omega,A,\psi)$ is a correction term, which is defined in Sect.~\ref{defofgamma}.
\end{itemize}
Since the analysis in the case of $\ch(F) \not=0$ has been discussed well (e.g. \cite{Wal03}), 
we do not need to worry about the existence of $\Gamma^\V(s,\pi,\omega,A,\psi)$.
Moreover, we have some properties of the $\gamma$-factor as in \cite{LR05}.
\begin{enumerate}
\item (multiplicativity): Let $W$ be a totally isotropic subspace of $V$, let $\W_1=(W^\bot/W,h)$. Then the stabilizer $P(W)$ of $W$ is a parabolic subgroup of
        $G(\V)$ which has the Levi subgroup isomorphic to $\GL(W) \times G(\W_1)$.
        Then, for irreducible representations $\sigma_0, \sigma_1$ of $\GL(W), G(\W_1)$, we have
        \[
        \gamma^\V(s,\pi\boxtimes\omega,\psi) 
        = \gamma_{\GL(W)}^{GJ}(s,\sigma_0\otimes\omega,\psi)
           \gamma_{\GL(W)}^{GJ}(s,\sigma_0^\vee\otimes\omega,\psi)
           \gamma^{\W_1}(s,\sigma_1\boxtimes\omega,\psi)
        \]
        if $\pi$ is isomorphic to a subquotient of 
        $\Ind_{P(W)}^{G(\V^\Box)}\sigma_0\boxtimes\sigma_1$. 
        Here, $\sigma_0^\vee$ is the contragredient representation of $\sigma_0$.
        $\gamma_{\GL(W)}^{GJ}(s, - ,\psi)$ is the $\gamma$-factor of Godement-Jacquet (see Sect.~\ref{mainthm}).
        \label{mlintro}
\item (global functional equation): Let $\F$ be a global field of $\ch(\F) \not=0,2$, 
        let $\E$ be a division $\F$-algebra with $[\E:\F]=1,4$,
         and let $\underline{\V}$ be an $\epsilon$-hermitian space over $\E$. 
        We denote $\A$ by the ring of adeles of $\F$.
         Thanks to the argument of Eisenstein series, 
        for an irreducible automorphic cuspidal representation $\Pi$ of $G(\underline{\V})(\A)$, 
        a Hecke character $\underline{\omega}$ of $\A^\times / \F^\times$, 
        and a non-trivial additive character $\underline{\psi}$ of $\A/\F$, we have
        \[
        L_S(s, \std\circ\phi_{\Pi\boxtimes\underline{\omega}})
        = \prod_{v \in S}
        \gamma_v^{\underline{\V}}(s, \Pi\boxtimes\underline{\omega},\underline{\psi}) \cdot 
            L_S(1-s, \std\circ\phi_{\Pi^\vee\boxtimes\underline{\omega}^{-1}})
        \]
        where $S$ is a finite set of places of $\F$ containing all places where either 
        $\underline{\V}, \Pi, \underline{\omega}$, or $\underline{\psi}$ is ramified (Sect.~\ref{gn}), 
        $\std$ is the standard homomorphism of ${}^L\!(G(\V))$ into $\GL_N$ 
        and $L_S(s, \std\circ\phi_{\Pi\boxtimes\omega})$ is the partial standard $L$-factor,
        which is the product of the local $L$-functions of \cite{GR10} over all places 
        not contained in $S$.
       \label{gfeintro}
\end{enumerate}
Put $\delta$ be the integer satisfying $\delta^2 = [E:Z(E)]$ where $Z(E)$ is the center of $E$.  Our main theorem (Theorem \ref{maingamma}) says that 
\begin{itemize}
\item if $G$ is anisotropic, the semisimple rank $r(G)$ of $G$ is at most $2/\delta$, 
        then we have \eqref{desiintro} with $\pi = 1_{\V}$ and $\omega=1$ where $1_{\V}$ denotes the trivial representation of $G(\V)$, $1$ denotes the trivial character of $F^\times$, $\phi_0$ denotes the principal parameter of $G(\V)$ (see Sect.~\ref{pr}).
\item the $\gamma$-factor $\gamma^\V(s,\pi\boxtimes\omega,\psi)$ is characterized by 
      \eqref{mlintro}, \eqref{gfeintro}, the equation \eqref{desiintro} 
       for unramified cases and for minimal cases (noted above), and some basic properties.
\end{itemize}
The latter is proved by combining \eqref{mlintro}, \eqref{gfeintro}, and the globalization results of \cite{Hen84, GL18}.
The former is proved by some explicit computations (Sect.~\ref{computation}). 
Note that we cannot avoid computing $\gamma^\V(s, 1_\V\boxtimes1, \psi)$ in the orthogonal cases of $\dim_FV = 3, 4$ and unitary cases of $\dim_EV =2$ because of the inability of techniques of archimedean places.

\subsection*{
                 Acknowledgments:
                  }

The author would like to thank Atsushi Ichino for suggesting the problem and for his supports and advice, and Teruhisa Koshikawa for comment about local Langlands correspondence over a local function field.


\section{
            Settings and notations
           }\label{setting}

In this section, we fix some notations. 

\subsection{
                Fields
                }

Let $F$ be a local field of characteristic $p \not=0,2$, and let $E$ be either $F$ itself, a $2$ dimensional semisimple $F$-algebra, or a quaternion algebra over $F$. Although our interest is primarily when $E$ is division, we allow $E$ to split (i.e. $E =F \times F$ or $E = \M_2(F)$) since they appear as a localization of global division algebras. 

Then, we fix some notations associated to the fields:
\begin{itemize}
\item $\varpi$ denotes a uniformizer of $F$, $\cO$ denotes the integer ring of $F$, 
        and $q$ denotes the cardinality of $\cO/\varpi \cO$,
\item for $a \in F^\times$, $\chi_a$ denotes the character of $F^\times$
         defined by $\chi_a(x) = (a, x)_F$ for $x \in F^\times$ 
         where $(\  ,\ )_F$ is the Hilbert symbol, 
\item when $[E:F]=2$, $\chi_E$ denotes a quadratic character of $F^\times$ associated 
         to $E/F$,
\item $\ord_F: F^\times \rightarrow \Z$ denotes the order of $F$,     
\item $Z(E)$ denotes the center of $E$,    
\item $*:E \rightarrow E: x \mapsto x^*$ denotes the main involution over $F$.
\item We also denote $*:F\times F \rightarrow F\times F: (a,b) \mapsto (b,a)$.
\item if $E$ is a division algebra, $\ord_E:E^\times \rightarrow \Z$ denotes the order of $E$, 
\item $T_{E/F}:E \rightarrow F$ denotes 
         \[ 
         \begin{cases}
         \mbox{the trace} & \mbox{if $[E:F] \leq 2$}, \\
         \mbox{the reduced trace} & \mbox{if $[E:F] = 4$},
         \end{cases}
         \]
\item $N_{E/F}:E^\times \rightarrow F^\times$ denotes
         \[
         \begin{cases}
         \mbox{the norm} & \mbox{if $[E:F] \leq 2$}, \\
         \mbox{the reduced norm} & \mbox{if $[E:F] = 4$},
         \end{cases}
         \]
\item $E_0 := \{ x \in E \mid T_{E/F}(x) = 0$, $E^1 := \{ x \in E^\times \mid N_{E/F}(x) = 1\}$, 
\item $| \cdot |_F$ and $| \cdot |_E$ denotes the normalized absolute values: i.e. 
        \[
         |\varpi|_F = q^{-1}, \ | \cdot |_E := | N_{E/F}( \cdot )|_F.
        \]
\item $| \cdot |: F \times F \rightarrow \R$ is a norm on $F \times F$ defined by 
        $|(a, b)| = |a|_F|b|_F$ for $a,b \in F$.
\item for any finite extension field $F'$ of $F$, we denote by $\zeta_{F'}(s)$ the
         local zeta function of $F'$: 
         \[
          \zeta_{F'}(s) = \frac{1}{1-{q'}^{-s}} 
         \]
         where $q'$ is the cardinality of the residue field of $F'$.
\end{itemize}

\subsection{
                Groups
                }\label{grps}

Let $\V = (V,h)$ where $V$ is a free right $E$ module of rank $n$ and $h$ is a map $V \times V \rightarrow E$ such that 
        \begin{itemize}
        \item $h$ is either $0$ or non-degenerate,
        \item $h(x, ya + zb) = h(x,y)a + h(x,z)b$ 
                             for $x,y \in V$ and $a,b \in E$, 
        \item there is an $\epsilon = \pm1$ such that $h(y, x) = \epsilon h(x,y)^*$
                             for $x,y \in V$.
        \end{itemize}
We denote by $G(\V)$ the algebraic group
\[
\{ g \in \GL(V) \mid h(gx,gy) = h(x,y) \mbox{ for all $x,y \in V$}\},
\]
and we denote by $G(\V)^\circ$ the Zariski connected component of $G(\V)$. We also write $G(\V)$ (resp. $G(\V)^\circ$) for the group of its $F$-rational points $G(\V)(F)$ (resp. $G(\V)^\circ(F)$) if there is no confusion. 
We will refer to the cases in which
\[
\begin{cases}
\mbox{$E=F$ and $h=0$} & \mbox{as $(\GL)$},  \\
\mbox{$E=F$, $h$ is non-degenerate and $\epsilon = 1$} & \mbox{as $(\SO)$}, \\
\mbox{$E=F$, $h$ is non-degenerate and $\epsilon = -1$} & \mbox{as $(\Sp)$}, \\
\mbox{$[E:F]=2$ and $h=0$} & \mbox{as $(\rqGL)$}, \\
\mbox{$[E:F] =2$, $h$ is non-degenerate and $\epsilon = 1$} & \mbox{as $(\rU)$}, \\
\mbox{$[E:F] =4$ and $h=0$} & \mbox{as $(\rQGL)$}, \\
\mbox{$[E:F] = 4$, $h$ is non-degenerate and $\epsilon = 1$} & \mbox{as $(\rQ_1)$}, \\
\mbox{$[E:F] = 4$, $h$ is non-degenerate and $\epsilon = -1$} & \mbox{as $(\rQ_{-1})$}.
\end{cases}
\]

For a totally isotropic subspace (if $h=0$, this means just a subspace) $W$ of $V$, we denote by $P_{G(\V)}(W)$ the parabolic subgroup stabilizing $W$. We write it $P(W)$ if there is no confusion. We denote by $U(W)$ the unipotent radical of $P(W)$.
We denote $\W_0 = (W, 0)$ and $\W_1 = (W^\bot/W, h)$, where 
\[
W^\bot = \{ v \in V \mid h(v,x) = 0 \mbox{ for all $x\in W$} \}.
\]
Then there is the exact sequence
\[
1 \rightarrow U(W) \rightarrow P(W) \rightarrow \GL(W) \times G(\W_1) \rightarrow 1
\]
and any Levi subgroup of $P(W)$ is isomorphic to
\[
\GL(W)\times G(\W_1).
\]

Since we use the globalization, we must consider the split cases: i.e. the cases $E = F \times F$ or $\M_2(F)$. Put 
\[
e := \begin{cases} (1,0) & \mbox{in the case $E = F\times F$}, \\
      \begin{pmatrix} 1 & 0 \\ 0 & 0 \end{pmatrix} & \mbox{in the case $E = \M_2(F)$}.
      \end{cases}
\]
Then, we define $\V^\natural = (V^\natural, h^\natural)$ consisting of the vector space $V^\natural = Ve$ over $F$ and a map $h^\natural: V^\natural \times V^\natural \rightarrow F$ so that 
\[
\begin{cases} h^\natural = 0 & \mbox{in the case $(\rU)$}, \\
            h(xe, ye) = \begin{pmatrix} 0 & 0 \\ h^\natural(xe,ye) & 0 \end{pmatrix} \ (x,y \in V)
                                       & \mbox{in the cases $(\rQGL), (\rQ_1), (\rQ_{-1})$}.
\end{cases}
\]
Moreover in the case $(\rqGL)$, we define ${\V^\natural}'= ({V^\natural}', {h^\natural}')$ consisting of the vector space ${V^\natural}' = V e^*$ over $F$ and the zero-map ${h^\natural}': {V^\natural}'\times{V^\natural}' \rightarrow F$. 
Then, we have isomorphisms 
\[
\begin{cases}
\iota: G(\V) \rightarrow G(\V^\natural) \times G({\V^\natural}'): g \mapsto (g|_{V^\natural}, g|_{{V^\natural}'})
& \mbox{in the case $(\rqGL)$}, \\
\iota: G(\V) \rightarrow G(\V^\natural): g \mapsto g|_{V^\natural}
& \mbox{in the other cases}.
\end{cases}
\]

We fix some notations associated to $\V$:
\begin{itemize}
\item $\disc(\V)$ denotes the discriminant of $\V$,
\item $F_\V$ denotes the field $F(\sqrt{d})$ where $d \in F^\times$ such that $\disc(\V)$ agree with the image of $d$,
\item in the case $(\SO)$, $c(\V)$ denotes the Hasse invariant of $\V$,
\item in the case $(\rU)$, $\epsilon(\V)$ denotes the symbol defined by $\chi_E(\disc(\V))$.
\end{itemize}


\subsection{
                Explicit notations
                }

Let $R \in \GL_n(E)$ so that ${}^t\!R^* = \epsilon R$. We define the map 
\[
\langle R \rangle: E^n \times E^n \rightarrow E: ({}^t\!(x_1, \ldots, x_n), {}^t\!(y_1, \ldots, y_n))
\mapsto (x_1, \ldots, x_n)^* \cdot R \cdot {}^t\!(y_1, \ldots, y_n).
\]
Then, the pair $(E^n, \langle R \rangle)$ satisfies the condition of Sect.~\ref{grps}, and the group $G((E^n, \langle R \rangle))$ is realized as the subgroup
\[
\{ g \in \GL_n(E) \mid {}^t\!g^* \cdot R \cdot g = \epsilon R \}
\]
of $\GL_n(E)$ where $N:\GL_n(E) \rightarrow Z(E)^\times$ is the reduced norm over $Z(E)$. Note that 
\[
\disc((E^n, \langle R \rangle)) = N(R) \times 
\begin{cases}
(-1)^{\frac{1}{2}n(n-1)}2^{-n} & \mbox{in the cases $(\SO)$, $(\Sp)$,}\\
(-1)^{\frac{1}{2}n(n-1)} & \mbox{in the case $(\rU)$,}\\
(-1)^n & \mbox{in the cases $(\rQ_1)$, $(\rQ_{-1})$.}
\end{cases}
\]


\subsection{
                Representations
                }\label{reps}

In this paper, by ``a representations of a reductive group'' we mean a complex smooth representation of the group. Let $\pi$ be a representation of $G(\V)$. Note that we often identify the representation space of $\pi$ with $\pi$ itself. We denote by $\pi^\vee$ the contragredient representation. Moreover, in the cases $(\GL), (\rqGL), (\rQGL)$ we denote by $\pi^\bigstar$ the contragredient representation of
\[
\begin{cases}
\pi & \mbox{ in the cases $(\GL), (\rQGL)$}, \\
\pi^* & \mbox{ in the case $(\rqGL)$}
\end{cases}
\]
where $\pi^*(g) := \pi(g^*)$ for $g \in \GL_n(E)$.

For a parabolic subgroup $P$ of $G(\V)$ and a representation $\sigma$ of the Levi subgroup of $P$, we denote by $\Ind_P^{G(\V)}\sigma$ the normalized induced representation: i.e. $\Ind_P^{G(\V)}\sigma$ is the representation by the right translation on the space of smooth functions $f : G(\V) \rightarrow \sigma$ satisfying
\[
 f(pg) = \delta_P(p)^{\frac{1}{2}}\sigma(p)(f(g)) 
\]
for $p \in P, g \in G(\V)$ where $\delta_P$ is the modular function of $P$.

Now suppose that $E$ is split. 
In the cases $(\rU), (\rQ_1), (\rQ_{-1}), (\rQGL)$, for a representation $\pi$ of $G(\V)$, we denote by $\pi^\natural$ the representation of $G(\V^\natural)$ satisfying $\pi = \pi^\natural\circ \iota$. 
In the case $(\rqGL)$, for a representation $\pi$ of $G(\V)$, we denote by $\pi^\natural$ the representation of $G(\V^\natural)$ and by ${\pi^\natural}'$ the representation of $G({\V^\natural}')$ satisfying $(\pi^\natural \boxtimes {\pi^\natural}')\circ\iota \cong \pi$.


\subsection{
                The standard local factors
                }\label{std_loc}

In this subsection, we explain the standard homomorphism of the complex $L$-dual group of 
\begin{align}\label{grps}
\begin{cases}
G(\V)^\circ \times \GL_1 & \mbox{in the cases $(\SO), (\Sp), (\rQ_1), (\rQ_{-1})$}, \\
G(\V) \times \Res_{E/F}(\GL_1) & \mbox{in the cases $(\rU), (\rqGL)$}, \\
G(\V) \times \GL_1\times \GL_1 & \mbox{in the cases $(\GL), (\rQGL)$},
\end{cases}
\end{align}
and we explain the standard local factors of $L$-parameters of them. Note that in the case $(\GL)$, the necessity of twisting by $\GL_1 \times \GL_1$ is caused when one consider local components of a globalization of $G(\V) \times \Res_{E/F}(\GL_1)$ in the case $(\rU)$. 

Put
\[
N = \begin{cases}
       2n  & \mbox{in the case $(\GL), (\rU), (\rQ_{-1})$}, \\
       2\lfloor \frac{n}{2} \rfloor & \mbox{in the case $(\SO)$}, \\
       n+1 & \mbox{in the case $(\Sp)$}, \\
        4n & \mbox{in the case $(\rqGL), (\rQGL)$}, \\
      2n+1 & \mbox{in the case $(\rQ_1)$}.
      \end{cases}
\]
In this paper, by standard homomorphisms we mean the homomorphisms defined as follows:
\begin{itemize}
\item  In the cases $(\GL), (\rQGL)$, we have ${}^L(G(\V)\times\GL_1\times\GL_1) = (\GL_{\frac{N}{2}}(\C)\times\C^\times\times\C^\times)\times W_F$, and we define $\std$ by
         \[
         \std: (\GL_{\frac{N}{2}}(\C)\times\C^\times\times\C^\times) \times W_F 
                \rightarrow \GL_{N}(\C): 
          (g, z_1, z_2, w) \mapsto \begin{pmatrix} z_1g & 0 \\ 0 & z_2\:{}^t\!g^{-1} \end{pmatrix}.
         \]
\item In the cases $(\SO, \mbox{$n$: even}), (\Sp), (\rQ_1), (\rQ_{-1})$, 
        we have ${}^L\!(G(\V)^\circ\times\GL_1) =  (\SO_N(\C)\times\C^\times) \rtimes W_F$, 
        and we define $\std$ by
         \[
         \std: (\SO_N(\C)\times\C^\times) \rtimes W_F \rightarrow \GL_N(\C):
          (g, z, w) \mapsto zg \cdot w_N
         \]
         where 
         \[
          w_N = \begin{cases} I_N & (w \in W_{F_{\V}}) \\
                             \begin{pmatrix} I_{N/2 - 1} & 0 & 0 & 0 \\ 
                                                  0 & 0 & 1 & 0 \\
                                                  0 & 1 & 0 & 0 \\
                                                  0 & 0 & 0 & I_{N/2-1} \end{pmatrix} 
                   & (w \in W_F \smallsetminus W_{F_{\V}}).
                   \end{cases}
         \]
\item In the case $(\SO, \mbox{$n$: odd})$, we have ${}^L\!(G(\V)^\circ\times\GL_1) 
         = (\Sp_{N}(\C)\times\C^\times) \times W_F$, and we define $\std$ by
         \[
         \std: (\Sp_{N}(\C) \times\C^\times) \times W_F \rightarrow \GL_N(\C):
                 (g, z, w) \mapsto zg
         \]
\item In the case $(\rU)$, we have ${}^L\!(G(\V)\times \Res_{E/F}\GL_1) 
         = (\GL_n(\C)\times\C^\times\times\C^\times)\rtimes W_F$, and we define $\std$ by
        \[
        \std: (\GL_n(\C)\times\C^\times\times\C^\times) \rtimes W_F \rightarrow \GL_N(\C): 
        (g, z_1, z_2, w) \mapsto \begin{pmatrix} z_1g & 0 \\ 0 & z_2J\:{}^t\!g^{-1}J^{-1}           
                                        \end{pmatrix}\widetilde{w}.
        \]
        where 
        \[
        J = \begin{pmatrix} & & & 1 \\ & & (-1) & \\ & \iddots & & 
             \\ (-1)^{n-1} & & &\end{pmatrix}, \ \ 
        \widetilde{w} = \begin{cases}
                            I_N & (w \in W_E) \\
                            \begin{pmatrix} 0 & I_{N/2} \\ -I_{N/2} & 0 \end{pmatrix} 
                            & (w \in W_F \smallsetminus W_E).
                            \end{cases}
        \]
\item In the case $(\rqGL)$, we have ${}^L\!(G(\V) \times \Res_{E/F}\GL_1)
         = (\GL_n(\C)\times\GL_n(\C)\times\C^\times\times\C^\times) \rtimes W_F$, 
        and we define $\std$ by
        \[
        \std: (\GL_n(\C)\times\GL_n(\C)\times\C^\times\times\C^\times) \rtimes W_F
               \rightarrow \GL_N(\C): (g_1, g_2, z_1, z_2, w) \mapsto 
               \begin{pmatrix} g_1z_1 & & & \\ 
                                   & g_2z_1 & & \\
                                   & & {}^t\!g_2^{-1}z_2 & \\
                                   & & & {}^t\!g_1^{-1}z_2 \end{pmatrix} \widetilde{w}.
        \]
\end{itemize}

Let $\phi$ be an $L$-parameter of the group $G(\V)^\circ$, and let $\psi:F\rightarrow \C^\times$ be a non-trivial additive character. Then we define the {\bf standard $\gamma$-factor} of $\phi$ by
\[
\gamma_F(s, \std\circ\phi, \psi) 
\]
where $\gamma_F(s, - , \psi)$ is the $\gamma$-factor in the sense of \cite{GR10}. Define the standard $L$- and $\epsilon$-factors in the same way.

\begin{rem}\label{Lpar orth}
We note here the local Langlands correspondence for classical groups (concluding orthogonal groups). 
Let $\pi$ be an irreducible representation of $G(\V)$ and let $\omega$ be a character of $F^\times$. We denote by $\pi^\circ$ an irreducible component of the restriction of $\pi$ to $G(\V)^\circ$. 
Consider the case $(\SO)$. Then the restriction of $\pi$ to $G(\V)^\circ$ may become reducible. If $\pi_1$ and $\pi_2$ are irreducible components of $\pi|_{G(\V)^\circ}$, considering the analogue of the argument in \cite[\S\S3.4-3.6]{AG17}, we can expect that
\[
\std\circ\phi_{\pi_1\boxtimes\omega}\cong \std\circ\phi_{\pi_2\boxtimes\omega}
\]
as representations of $W_F\times\SL_2(\C)$. Here we denotes by $\phi_{\pi_i\boxtimes\omega}$ the $L$-parameter of $\pi_i\boxtimes\omega$ for $i=1,2$.
\end{rem}

\begin{rem}
We allow $n$ to be zero. In the case $N=0$, we set $\gamma_F(s, \std\circ\phi, \psi) = 1$ for all $L$-parameter $\phi$.
\end{rem}

\subsection{
                The principal parameter
                }\label{pr}

In this subsection, we explain the principal parameter, which is expected to be associated with the Steinberg representation via the Langlands correspondence.

Let $G$ be a connected reductive group over $F$, let $T$ be a maximal torus of $\widehat{G}$, let $R(\widehat{G}, T)$ be the root system of $\widehat{G}$ with respect to $T$, let $\Delta$ be a positive system of $R(\widehat{G}, T)$, let $\Delta^1$ be a set of simple roots of $\Delta$, and for each $\alpha^\vee \in \Delta^1$, and let $X_{\alpha^\vee}$ be a fixed non-zero element of ${\rm Lie}({}^L\!G)$ such that ${\rm ad}(h)X_{\alpha^\vee} = \alpha^\vee(h)X_{\alpha^\vee}$ for $h \in {\rm Lie}(T)$. Put
\[
N_0 := \sum_{\alpha^\vee \in \Delta^1} X_{\alpha^\vee} \in {\rm Lie}(\widehat{G}) = {\rm Lie}({}^L\!G).
\]
Then, we deifne the {\bf principal parameter} as the $L$-homomorphism $\phi_0: \SL_2(\C)\times W_F \rightarrow {}^L\!G$ such that 
\[
d\phi_0(\begin{pmatrix} 0 & 1 \\ 0 & 0 \end{pmatrix}) = N_0 \in {\rm Lie}({}^L\!G), \ \phi_0((I_2, w)) = (1, w) \ (w \in W_F).
\]
One can show that the image of $\phi_0$ is not contained in any proper parabolic subgroup of ${}^L\!G$. Thus $\phi_0$ is an $L$-parameter of $G$.

\subsection{
                Global notations
                }\label{gn}

Let $\F$ be a global field of $\ch(\F) \not= 0,2$, let $\E$ be a division $\F$-algebra of $[\E:\F] =1,2,4$, and let $\underline{\V} = (\underline{V}, \underline{h})$ where $\underline{V}$ is an $n$-dimensional right $\E$ vector space and $\underline{h}$ is a map $\underline{V} \times \underline{V} \rightarrow \E$ such that 
\begin{itemize}
\item $\underline{h}$ is either $0$ or non-degenerate,
\item $\underline{h}(x, ya + zb) = \underline{h}(x, y)a + \underline{h}(x, z)b$ 
                             for $x,y,z \in \underline{V}$ and $a, b \in \E$, 
\item there is an $\epsilon = \pm1$ such that $\underline{h}(y, x) 
      = \epsilon \underline{h}(x, y)^*$ for $x,y \in \underline{V}$.
\end{itemize}
We say that $\underline{\V}$ is unramified at a place $v$ of $\F$ when $\E/\F$ is unramified at $v$ and
\begin{itemize}
\item $\V_v$ is isometric to the space
        \[
        \begin{cases}
        (\F_v^n, \langle 0 \rangle) & \mbox{in the case $(\GL)$}, \\
        (\E_v^n, \langle \diag(1, -1, \cdots, (-1)^{n-1}) \rangle)
                                                          &\mbox{in the case $(\SO), (\rU)$}, \\
        (\F_v^n, \langle \begin{pmatrix} 0 & I_{\frac{n}{2}} \\ -I_{\frac{n}{2}} & 0 \end{pmatrix}
                                              \rangle) & \mbox{in the case $(\Sp)$} 
        \end{cases}
        \]
        if $\E_v$ is a field,
\item $\V_v^\natural$ is unramified over $\F_v$ if $\E \not=\F$ and $\E_v$ is split over $\F_v$.
\end{itemize}
We denote by $\A$ the ring of adeles of $\F$,
and we denote by $G(\underline{\V})(\A)$ the restricted product $\prod_v' G(\underline{\V}_n)$ with respect to the canonical open compact subgroups $K_v$ for unramified places $v$.
We say that an irreducible automorphic representation $\Pi$ of $G(\underline{\V})(\A)$ is unramified at $v$ is unramified at $v$ if $\underline{\V}$ is unramified at $v$ and $\Pi_v$ has a non-zero $K_v$-fixed vector.
For a finite extension field $\F'$ over $\F$, we say that a Hecke character $\underline{\omega}$ of $(\A\otimes\F')^\times/{\F'}^\times$ is unramified at $v$ if $\F'/\F$ is unramified at $v$ and $\underline{\omega}_v = | \ |_{\F_v'}^{s_0}$ for some $s_0 \in \C$. 
We say that a non-trivial additive character $\underline{\psi}$ of $\A/\F$ is unramifed at $v$ if the order of $\underline{\psi}_v$ is zero.


\section{
            The doubling method
            }\label{dm}


\subsection{
                Doubled spaces and unitary groups 
                }\label{doubledsp}

Let $\V^\Box = (V^\Box,h^\Box)$ be a pair, where $V^\Box= V\times V$ and $h^\Box = h \oplus (-h)$, that is, the map defined by
\[
h^\Box((x_1,x_2),(y_1,y_2)) = h(x_1,y_1)-h(x_2,y_2)
\]
for $x_1,x_2,y_1,y_2 \in V$. Then $G(\V) \times G(\V)$ acts on $V \times V$ by 
\[
(g_1, g_2)\cdot (x_1,x_2) = (g_1x_1,g_2x_2), 
\]
so that $G(\V) \times G(\V)$ can be embedded naturally in $G(\V^\Box)$. Consider the maximal totally isotropic subspaces
\begin{align*}
V^\tru &= \{ (x,x) \in V^\Box \mid x \in V\}, \\
V^\trd & =\{ (x,-x) \in V^\Box \mid x \in V\}.
\end{align*}
Note that $V = V^\tru \oplus V^\trd$. Then $P(V^\tru)$ is a maximal parabolic subgroup of $G(\V^\Box)$ and its Levi subgroup is isomorphic to
\[
\begin{cases}
\GL(V^\tru) \times \GL(V^\Box/V^\tru) & \mbox{ in the cases $(\GL), (\rqGL), (\rQGL)$} \\
\GL(V^\tru) & \mbox{ in the other cases}.
\end{cases}
\]


\subsection{
                Zeta integrals and intertwining operators
                }\label{zeta}

In this subsection, we assume that $n \not=0$. Denote by $\Lambda_1: P(V^\tru) \rightarrow Z(E^\times)$ and $\Lambda_2: P(V^\tru) \rightarrow Z(E^\times)$ the characters given by
\[
\Lambda_1(g) = N_{V^\tru}(g), \ \Lambda_2(g) = N_{V^\Box/V^\tru}(g)
\]
Here, $N_{V^\tru}(g)$ ({resp. }$N_{V^\Box/V^\tru}(g)$) is the reduced norm of the image of $g$ in $\End_EV^\tru$ (resp. $\End_EV^\Box/V^\tru$) over $Z(E)$.  Moreover, we define
\[
\Delta_{V^\tru} = 
\begin{cases}
(\Lambda_1, \Lambda_2^{-1}): P(V^\tru) \rightarrow F^\times \times F^\times
& \mbox{in the cases $(\GL, \rQGL)$}, \\
\Lambda_1(\Lambda_2^*)^{-1}: P(V^\tru) \rightarrow E^\times 
& \mbox{in the case $(\rqGL)$}, \\
\Lambda_1: P(V^\tru) \rightarrow Z(E^\times) & \mbox{in the other cases}.
\end{cases}
\]

Let $\omega$ be a character of 
\[
\begin{cases}
F^\times \times F^\times & \mbox{in the cases $(\GL), (\rQGL)$}, \\
Z(E^\times) & \mbox{in the other cases}.
\end{cases}
\]
We denote by $\omega^*$ the character defined by the composition of $\omega$ and $*$.
For $s \in \C$, put $\omega_s = \omega \cdot | \cdot |^s$. Choose a compact subgroup $K$ of $G(\V^\Box)$ such that $G(\V^\Box) = P(V^\tru)K$. Denote by $I(s,\omega)$ the degenerate principal series representation
\[
\Ind_{P(V^\tru)}^{G(\V^\Box)}(\omega_s\circ \Delta_{V^\tru})
\]
consisting of smooth right $K$-finite functions $f:G(\V^\Box) \rightarrow \C$ satisfying
\[
 f(pg) = \delta_{P(V^\tru)}^{\frac{1}{2}}(p) 
          \cdot \omega_s(\Delta_{V^\tru}(p)) \cdot f(g) 
\]
for $p \in P(V^\tru)$ and $g \in G(\V^\Box)$, where $\delta_{P(V^\tru)}$ is the modular function of $P(V^\tru)$. 
We may extend $|\Delta_{V^\tru}|$ to a right $K$-invariant function on $G(\V^\Box)$ uniquely. For $f \in I(0,\omega)$, put $f_s = f \cdot |\Delta_{V^\tru}|^s \in I(s,\omega)$. Let $A_1$ be the center of $P(V^\tru)$. Define an intertwining operator $M(s,\omega): I(s,\omega) \rightarrow I(-s, (\omega^*)^{-1})$ by
\[
(M(s,\omega)f_s)(g) = \int_{U(V^\tru)}f_s(w_1ug) \:du
\]
where $w_1 \in G(\V^\Box)$ is a representative of the longest element of the Weyl group $W(A_1, G(\V^\Box))$ with respect to $A_1$. This integral defining $M(s,\omega)$ converges absolutely for $\Re s \gg 0$ and admits a meromorphic continuation to $\C$ (e.g. \cite[Theorem IV.1.1]{Wal03}). 

Let $\pi$ be an irreducible representation of $G(\V)$. For a matrix coefficient $\xi$ of $\pi$, and for $f \in I(\omega,0)$, define the zeta integral by
\[
Z^\V(f_s,\xi) = \int_{G(\V)}f_s((g,1))\xi(g) \:dg. 
\]
Then the zeta integral satisfies the following basic properties as in the characteristic zero case. 

\begin{thm}\label{zeta_int}
\begin{enumerate}
\item The integral $Z^\V(f_s,\xi)$ converges absolutely for $\Re s \gg 0$
         and extends to a meromorphic function in $s$. Moreover, 
         the function $Z^\V(f_s,\xi)$ is a rational function of $q^{-s}$. 
\item There is a meromorphic function $\Gamma^\V(s,\pi,\omega)$ such that 
\[
Z^\V(M(s, \omega)f_s,\xi) = \Gamma^\V(s,\pi,\omega) Z^\V(f_s,\xi)
\]
for all matrix coefficient $\xi$ of $\pi$ and $f_s \in I(s,\omega)$.
\end{enumerate}
\end{thm}

\begin{proof}
The boundedness of the function $|\Delta_{V^\tru}|$ on $U(V^\trd)$ in our case can be seen in \cite[Lemma II.3.4]{Wal03}.
Then the proof of \cite[Theorem 3]{LR05} still works. 
\end{proof}


\subsection{
                Normalization of the intertwining operators
                }\label{dwf}

In this subsection, we assume that $n\not=0$ unless stated otherwise. Note that $E$ is possibly split. We denote by $\u(V^\tru)$ the Lie algebra of $U(V^\tru)$. We regard $\u(V^\tru)$ as a subspace of $\End_E(V^\Box)$ and for $r=0, \ldots, n$ we denote by $\u(V^\tru)_r$ the set of $A \in \u(V)$ of rank $r$. We exclude the case $(\SO, n:\mbox{odd})$ for a while so that $\u(V^\tru)_n \not= \varnothing$.

Fix a non-trivial additive character $\psi: F \rightarrow \C^\times$ and $A \in \u(V^\tru)_n$. We define
\[
\psi_A: U(V^\trd) \rightarrow \C^\times: u \mapsto \psi(\operatorname{Tr}_{Z(E)/F}\circ\operatorname{Tr}_{V^\Box}(uA))
\]
where $\operatorname{Tr}_{V^\Box}$ denotes the reduced trace of $\End_E(V^\Box)$ over $Z(E)$. For $f \in I(\omega,0)$ we define
\[
l_{\psi_A}(f_s) = \int_{U(V^\trd)}f_s(u)\psi_A(u) \:du.
\]
Then the integral defining $l_{\psi_A}$ converges for $\Re s \gg 0$ and admits a holomorphic continuation to $\C$ (\cite[\S3.3]{Kar79}). 

The functional $l_{\psi_A}$ is called a \emph{degenerate Whittaker functional}, which is a $(U(V^\trd),\overline{\psi_A})$-equivariant functional on $I(s,\omega)$. 
On the other hand, the space of $(U(V^\trd),\overline{\psi_A})$-equivariant functionals on $I(s,\omega)$ is one dimensional for all $s \in \C$ (\cite[Theorem 3.2]{Kar79}).
Hence we have the following proposition.

\begin{prop}\label{def_of_c}
There is a meromorphic function $c(s,\omega,A,\psi)$ of $s$ such that 
\[
l_{\psi_A} \circ M(s,\omega) = c(s,\omega,A,\psi) l_{\psi_A}.
\]
\end{prop}

Then we define the normalized intertwining operator 
\[
M^*(s,\omega,A,\psi) = c(s,\omega,A,\psi)^{-1}M(s,\omega)
\]
and put
\[
\Gamma^\V(s,\pi,\omega,A,\psi) = c(s,\omega,A,\psi)^{-1} \Gamma^\V(s,\pi,\omega).
\]
Clearly, $\Gamma^\V(s,\pi,\omega,A,\psi)$ is a meromorphic function of $s$ satisfying
\begin{align}\label{Ga*}
Z(M^*(s,\omega,A,\psi)f_s,\xi) = \Gamma^\V(s,\pi,\omega,A,\psi)Z(f_s,\xi)
\end{align}
for any $f \in I(\omega,0)$ and any coefficient $\xi$ of $\pi$. 

\begin{rem}
In the case $n=0$, we set $\Gamma^\V(s,\pi, \omega, A, \psi) = 1$.
\end{rem}

\begin{rem}
Thanks to the normalization, the function $\Gamma^\V(s, \pi, \omega, A, \psi)$ does not depend on the choices of $K, \omega_1$, and measures on $G(\V), U(V^\tru), U(V^\trd)$.
\end{rem}


\subsection{
                Odd orthogonal cases
                }\label{oddorth}

In the case $(\SO, n:\mbox{odd})$,  the set $\u(V^\tru)_n$ is empty. 
However, Lapid-Rallis found an appropriate normalization of the intertwining operator $M(s,\omega)$, which yields the $\gamma$-factor in this case \cite[\S6]{LR05}.
We follow their construction.

Fix a nontrivial additive character $\psi: F \rightarrow \C^\times$ and $A \in \u(V^\tru)_{n-1}$.
We define $\psi_A$ and $l_{\psi_A}$ as Sect.~\ref{dwf}. Moreover, we define another Whittacker functional $l_{\psi_{A,L}}'$ as follows.
We denote by $p^\trd$ the projection of $V^\Box$ onto $V^\trd$ along the decomposition $V^\Box = V^\tru \oplus V^\trd$. 
Let $K= \ker A \cap V^\trd$, let $L$ be an anisotropic line of $V^\Box$ such that $p^\trd(L) = K$, and let $\sigma_L \in G(\V^\Box)$ be the orthogonal reflection around $L$.
We can identify $U(V^\trd) \cap \sigma_LP(V^\tru)\sigma_L$ with $(V^\tru\cap K^\bot)\otimes K$ as in \cite[Lemma 11]{LR05}.
We define a bilinear form
\[
h_{A,L}:(V^\tru\cap K^\bot)\otimes K \times (V^\tru\cap K^\bot)\otimes K \rightarrow F
\]
by
\[
h_{A,L}(x\otimes a, y\otimes b) = h^\Box(x',y)h^\Box(\sigma_L(a), b)
\]
for $x,y \in V^\tru\cap K^\bot$ and $a,b \in K$, where $x' \in V^\tru$ with $Ax' = x$.
Then we fix the self dual Haar measure $d''u$ on $U(V^\trd) \cap \sigma_LP(V^\tru)\sigma_L$ with respect to $h_{A,L}$.
Now, we define the functional $l_{\psi_{A,L}}'$ by
\[
l_{\psi_{A,L}}'(f_s) := \int_{U(V^\trd) \cap \sigma_LP(V^\tru)\sigma_L \backslash U(V^\trd)}
                      f_s(\sigma_L(u)) \psi_A(u) \: d'u
\]
where $d'u$ is the Haar measure on $U(V^\trd) \cap \sigma_LP(V^\tru)\sigma_L \backslash U(V^\trd)$ which is  compatible with $du$ on $U(V^\trd)$ and $d''u$ on $U(V^\trd) \cap \sigma_LP(V^\tru)\sigma_L$. Then the integral defining $l_{\psi_{A,L}}'$ converges for $\Re s \gg 0$ and admits a holomorphic continuation to $\C$.

\begin{prop}
There is a meromorphic function $c(s,\omega,A,L,\psi)$ of $s$ such that 
\[
l_{\psi_{A,L}}' \circ M(s,\omega) = c(s,\omega,A,L,\psi) l_{\psi_A}
\]
\end{prop}
Then we define the normalized intertwining operator
\[
M^*(s,\omega,A,L,\psi) = c(s,\omega,A,L,\psi)^{-1} M(s,\omega)
\]
and put
\[
\Gamma^\V(s,\pi,\omega,A,L,\psi) = c(s,\omega,A,L,\psi)^{-1}\Gamma^\V(s,\omega)
\]
Clearly, $\Gamma^\V(s,\pi,\omega,A,L,\psi)$ is a meromorphic function of $s$ satisfying
\begin{align}\label{oddGa*}
Z(M^*(s,\omega,A,L,\psi)f_s, \xi) = \Gamma^\V(s,\pi,\omega,A,L,\psi) Z(f_s, \xi)
\end{align}
for any $f \in I(\omega,0)$ and any coefficient $\xi$ of $\pi$. 

\begin{rem}
As in the other cases, the function $\Gamma^\V(s, \pi, \omega, A, \psi)$ does not depend on the choices of $K, \omega_1$, and measures on $G(\V), U(V^\tru), U(V^\trd)$.
\end{rem}


\section{
            Statement of the main theorem
            }\label{sm}

\subsection{
                Definition of the $\gamma$-factor
                }\label{defofgamma}

In the cases other than $(\SO, n:\mbox{odd})$, we can define the invariants $N_V(A)$ and $\disc(A)$ for $A \in \u(V^\tru)_n$ as follows.
The nilpotent element $A$ defines an isomorphism from $V^\trd$ to $V^\tru$. Consider an endomorphism $\varphi_A \in \End_E(V)$ so that the following diagram is commutative:
\[
\xymatrix{
              V^\trd \ar[r]^-{A} \ar[d]_p & V^\tru \ar[d]^p \\
                    V  \ar[r]_-{\varphi_A} & V 
              }
\]
where $p$ is the first projection of $V\times V$ to $V$. Then we define 
\begin{align*}
N_V(A) := N_V(\varphi_A), \ \ \disc(A) := (-1)^n N_V(A) \in F^\times/F^{\times 2}.
\end{align*}
In the case $n=0$, we set $\disc(A) = N_V(A)=1$.

In the case $(\SO, n:\mbox{odd})$, we define the invariant $N_V(\widetilde{A}_L)$ for $A \in \u(V^\tru)_{n-1}$ and for $L \subset V^\Box$ of Sect.~\ref{oddorth} as follows.
We have the decomposition 
\[
V^\trd = K \oplus (V^\trd \cap {K'}^\bot)
\]
where $K' = \sigma_L K$.
We can choose an isomorphism $\widetilde{A}_L: V^\trd \rightarrow V^\tru$ so that 
\begin{itemize}
\item $\widetilde{A}_L|_{K} = \sigma_L|_{K}$, 
\item $\widetilde{A}_L|_{{K'}^\bot \cap V^\trd} = A |_{{K'}^\bot \cap V^\trd}$.
\end{itemize}
Consider an endomorphism $\varphi_{\widetilde{A}_L} \in \End_E(V)$ so that the following diagram is commutative:
\[
\xymatrix{
              V^\trd \ar[r]^-{\widetilde{A}_L} \ar[d]_p & V^\tru \ar[d]^p \\
                    V  \ar[r]_-{\varphi_{\widetilde{A}_L}} & V 
              }
\]
where $p$ is the first projection of $V\times V$ to $V$. Then we can define
\begin{align*}
N_V(\widetilde{A}_L) := N_V(\varphi_{\widetilde{A}_L}).
\end{align*} 

\begin{df}\label{dfgm}
Let $\pi$ be an irreducible representation of $G(\V)$, let $\omega$ be a character of 
\[
\begin{cases}
F^\times \times F^\times & \mbox{in the cases of $(\GL), (\rQGL)$}, \\
Z(E^\times) & \mbox{in the other cases},
\end{cases}
\] 
and let $\psi$ be a non-trivial character of $F$. 
\begin{enumerate}
\item In the cases other than $(\SO, n:\mbox{odd})$, we define the $\gamma$-factor of $\pi$ by
        \[
        \gamma^\V(s+\frac{1}{2},\pi\boxtimes\omega,\psi)
           = \Gamma^\V(s,\pi,\omega,A,\psi) \cdot c_\pi(-1) \cdot R(s,\omega,A,\psi)
        \]
        where $\Gamma^\V(s,\pi,\omega,A,\psi)$ is the meromorphic function defined by \eqref{Ga*}, $c_\pi$ is the central character of $\pi$, and
        \[
        R(s,\omega,A, \psi) 
        = \begin{cases}
         \omega_s(N_V(\frac{1}{2}A), N_V(-\frac{1}{2}A))^{-1} 
                                                                  & \mbox{in the cases $(\GL), (\rQGL)$}, \\
         \omega_s(N_V(\frac{1}{2}A))^{-1}\omega_s(N_V(-\frac{1}{2}A)^*)^{-1}
                                                                   & \mbox{in the case $(\rqGL)$}, \\
         \omega_s(N_V(A))^{-1} \gamma(s+\frac{1}{2}, \omega\chi_{\disc(A)},\psi)
         \epsilon(\frac{1}{2}, \chi_{\disc(A)}, \psi)^{-1} & \mbox{in the cases $(\Sp), (\rQ_1)$}, \\
         \omega_s(N_V(A))^{-1}\epsilon(\frac{1}{2}, \chi_{\disc(\V)}, \psi)
                   & \mbox{in the cases $(\SO, n:\mbox{even}), (\rQ_{-1})$}, \\ 
         \omega_s(N_V(A))^{-1}\epsilon(\V) & \mbox{in the case $(\rU)$}.
          \end{cases}
        \]
\item In the case $(\SO, n:\mbox{odd})$, we define the $\gamma$-factor of $\pi$ by 
        \[
        \gamma^\V(s+\frac{1}{2}, \pi\boxtimes\omega, \psi)
            = \Gamma^\V(s,\pi,\omega,A, L, \psi) 
              \cdot \omega_s(N_V(\widetilde{A}_L))^{-1} \cdot c(\V).
        \]
        where $\Gamma^\V(s,\pi,\omega,A,L,\psi)$ is the meromorphic function defined by \eqref{oddGa*} and $c(\V)$ is the Hasse invariant attached to $\V$ (cf. \cite{Sch85}).
\end{enumerate} 
\end{df}

\begin{rem}
This agrees with the analogue of the corrected definition of $\gamma$-factor over a local field of characteristic zero (for a precise discussion, see \cite[\S5.3]{Kak19}).
\end{rem}

\begin{lem}
The definition of $\gamma^\V(s,\pi\boxtimes\omega, \psi)$ does not depend on the choice of $A$ (and $L$ in the case $(\SO, n:\mbox{odd})$).
\end{lem}

\begin{proof}
Let $M$ be the Levi subgroup $P(V^\tru) \cap P(V^\trd)$ of $P(V^\tru)$. Note that $M$ is isomorphic to $\GL_n(E)$.
At first, we consider the cases other than $(\SO, n:\mbox{odd}), (\Sp, n:\mbox{even}), (\rQ_1)$.
In these cases, the adjoint action of $M$ on $\u(V^\tru)_n$ is transitive. Take $A, A' \in \u(V^\tru)_n$ and $m \in M$ so that $A' = mAm^{-1}$. Then, we have
\[
c(s,\omega, A', \psi) = \omega_s(\Delta_{V^\tru}(m^{-1}))\omega_s(\Delta_{V^\tru}(m^{-1})^*)\cdot c(s,\omega,A,\psi)
\]
(\cite[Lemma 10]{LR05}). On the other hand, we have
\[
\begin{cases}
N_V(A') = \Delta_{V^\tru}(m)\Delta_{V^\tru}(m)^* N_V(A) &  (\rU), (\rqGL) \\
N_V(A') = \Lambda_1(m)\Lambda_2^{-1}(m)N_V(A) & (\GL), (\rQGL)
\end{cases}
\]
Hence, we have the lemma.

Second, in the case $(\SO, n:\mbox{odd})$, the Levi subgroup $M$ acts transitively on the set
\[
\mathfrak{X} = \{ (A, L) \mid A \in \u(V^\tru)_{n-1}, \ L: \mbox{anisotropic line in $V^\Box$ },
                  \ p^\trd(L) =V^\trd \cap \ker A\}  
\]
by $m\cdot(A,L) = (mAm^{-1}, mL)$. Let $(A,L), (A',L') \in \mathfrak{X}$ and $m \in M$ so that $(A',L') = m\cdot (A,L)$. Then, we have
\[
c(s,\omega,A',L',\psi) = \omega_s(\Delta_{V^\tru}(m))^{-2} c(s, \omega,A,L,\psi)
\]
(\cite[Lemma 12]{LR05}). On the other hand, we have
\[
N_V(\widetilde{A_{L'}'}) = \omega_s(\Delta_{V^\tru}(m))^2 N_V(\widetilde{A_L}).
\]
Hence, we have the lemma.

Finally, in the cases $(\Sp, n:\mbox{even}), (\rQ_1)$, the adjoint action of $M$ on $\u(V^\tru)_n$ is not transitive. However, there are explicit formulas of $c(s,\omega,A,\psi)$ in these cases \cite{Ike17} \cite{Yam17}, and we have the lemma (see \cite[p.9]{Kak19}).
\end{proof}


\subsection{
                Minimal cases
                }\label{mincases}

Before stating the main theorem, we list the ``minimal cases'';
\begin{itemize}
\item The trivial cases: 
        $(\SO, n=0,1), (\Sp, n=0), (\rU, n=0), (\rQ_1, n=0), (\rQ_{-1}, n=0)$;
\item The anisotropic orthogonal cases: $(\rSOa, n=2,3,4)$;
\item The one dimensional quaternionic unitary cases equipped with a division quaternion algebra: $(\rQ_{-1},n=1), (\rQ_1, n=1)$;
\item The one dimensional unitary cases equipped with a quadratic extension field: $(\rU, n=1)$;
\item The two dimensional anisotropic unitary cases equipped with a ramified quadratic extension field: $(\rUra, n=2)$.
\end{itemize}


\subsection{
                Main theorem
                }
                \label{mainthm}

 For a non-trivial character $\psi$ of $F$ and an irreducible representation $\rho$ of $\GL_m(D)$ where $D$ is the division central algebra over $F$, we can attach the ``Godement-Jacquet $\gamma$-factor'' as
\[
\gamma^{GJ}(s,\rho, \psi)
 = \varepsilon^{GJ}(s,\rho,\psi) \frac{L^{GJ}(1-s,\rho^\vee)}{L^{GJ}(s,\rho)}.
\]
where $L^{GJ}(s,\rho)$ (resp. $\varepsilon^{GJ}(s,\rho,\psi)$) is the $L$-factor (resp. the $\varepsilon$-factor) defined in \cite[Theorems 3.3,8.7]{GJ72}.

Now we state our main theorem:

\begin{thm}[Main]\label{maingamma}
The factor $\gamma^\V(s,\pi\boxtimes\omega,\psi)$ satisfies the following properties:
\begin{enumerate}
 \item (unramified twisting)
         \[
         \gamma^\V(s,\pi\boxtimes\omega_{s_0},\psi) = \gamma^\V(s+s_0,\pi\boxtimes\omega,\psi)
         \]\label{ut}
          for $s_0 \in \C$.
 \item (multiplicativity)
         Let $W$ be a totally isotropic subspace of $V$, and let $\sigma = \sigma_0 \boxtimes\sigma_1$ be an irreducible representation of $\GL(W) \times G(\W_1)$ (Sect.~\ref{grps}). 
         If $\pi$ is isomorphic to a subquotient of $\Ind_{P(W)}^G(\sigma)$, then
         \[
         \gamma^\V(s,\pi\boxtimes\omega,\psi) 
         = \gamma^{\W_0}(s,\sigma_0\boxtimes\omega,\psi)
            \gamma^{\W_1}(s,\sigma_1\boxtimes\omega,\psi).
         \] \label{ml}
 \item (unramified cases)
         Suppose that $E$ is $F$ itself or an unramified quadratic extension field of $F$. 
         If $\V, \pi, \omega$ and $\psi$ are unramified in the sense of Sect.~\ref{gn}, then
         \[
         \gamma^\V(s,\pi\boxtimes\omega, \psi) = 
         \frac{L(1-s, \pi^\vee\boxtimes\omega^{-1}, \std)}{L(s,\pi\boxtimes\omega, \std)}.
         \]
         where $L(s, -, \std)$ is the standard $L$-function (see Sect.~\ref{std_loc}). \label{unra}
 \item (split cases)
         Suppose that $E$ is split, then
         \[
         \gamma^\V(s,\pi\boxtimes\omega,\psi) = 
         \begin{cases}
         \gamma^{\V^\natural}(s,\pi^\natural\boxtimes\omega_1\boxtimes\omega_2,\psi)
         \gamma^{{\V^\natural}'}(s,{\pi^\natural}'\boxtimes\omega_2\boxtimes\omega_1, \psi)
         & \mbox{in the case $(\rqGL)$}, \\
         \gamma^{\V^\natural}(s,\pi^\natural\boxtimes\omega,\psi) 
         & \mbox{in the other cases},
         \end{cases}
         \]
         where $\omega_1$ and $\omega_2$ are the character of $F^\times$ such that 
         $\omega_1\boxtimes\omega_2 
         = \omega: F^\times\times F^\times\rightarrow \C^\times$
         in the case $(\rqGL)$. 
         \label{split}
 \item (functional equation)
         \[
         \gamma^\V(s,\pi\boxtimes\omega,\psi)
          \gamma^\V(1-s,\pi^\vee\boxtimes\omega^{-1},\psi^{-1}) = 1.
         \]\label{fe}
 \item (self duality)
         \[
         \gamma^\V(s,\pi^\vee\boxtimes\omega,\psi) 
          = \gamma^\V(s,\pi\boxtimes\omega^*,\psi).
         \]\label{sd}
 \item (dependence on $\psi$)
         Denote by $\psi_a$ the additive character $x \mapsto \psi(ax)$ of $F$
         for $a \in F^\times$. Then
         \[
         \gamma^\V(s,\pi\boxtimes\omega,\psi_a)
               = T_N(s,\omega,a) \cdot \gamma^\V(s,\pi\boxtimes\omega,\psi)
         \]
        where
        \[
               T_N(s,\omega,a)
                  =\begin{cases}
                 \omega_{s-\frac{1}{2}}(a)^{\frac{N}{2}}
                       & \mbox{ in the cases $(\GL)$, $(\rqGL)$, $(\rQGL)$}, \\
                 \omega_{s-\frac{1}{2}}(a)^N 
                       &  \mbox{ in the cases , $(\Sp)$, $(\rU)$, $(\rQ_1)$}, \\
                 \omega_{s-\frac{1}{2}}(a)^N\chi_{\disc(\V)}(a) 
                       &  \mbox{ in the cases $(\SO)$, $(\rQ_{-1})$}.
                 \end{cases}
         \]\label{psi}
 \item (minimal cases)
          Suppose that $\V$ is a minimal case in the sense of Sect.~\ref{mincases}. Then
          \[
          \gamma^\V(s,1_\V\boxtimes1,\psi) = \gamma_F(s, \std\circ\phi_0, \psi)
          \]
          where $1_\V$ is the trivial representation of $G(\V)$, $1$ is the trivial character 
          of $Z(E^\times)$, and the right hand side is the $\gamma$-factor in the sense of
          \cite{GR10}.
         \label{min}
 \item ($\GL_n$-factors) 
          In the cases $(\GL), (\rqGL), (\rQGL)$, 
          \[
          \gamma^\V(s,\pi\boxtimes\omega,\psi) 
            = \gamma^{GJ}(s,\pi\otimes\omega_1,\psi)
               \gamma^{GJ}(s,\pi^\bigstar\otimes\omega_2,\psi).
          \] 
          where $\pi^\bigstar$ is the representation of $G(\V)$ defined in Sect.~\ref{reps}, and
         \[
           \begin{cases}
             \omega = \omega_1\boxtimes\omega_2: F^\times\times F^\times\rightarrow \C^\times. & \mbox{ in the cases $(\GL), (\rQGL)$}, \\
             \omega_1 = \omega_2 = \omega: E^\times \rightarrow \C^\times 
              &\mbox{ in the case $(\rqGL)$}.
          \end{cases} 
         \]\label{GLn}
 \item (global functional equation)
         Let $\F$ be a global field, let $\E$ be a division algebra over $\F$ 
         of $[\E:\F] = 1,2,4$, and let $\underline{\V}$ be one of the spaces of Sect.~\ref{gn} over 
        $\E$.  Let $\Pi$ be an irreducible cuspidal automorphic representation of
        $G(\underline{\V})(\A)$. Then for finite set $S$ of places of $\F$ containing all places
        where $\underline{\V}$ and $\Pi$ are ramified, the functional equation
         \[ 
          L_S(s, \std\circ\phi_{\pi\boxtimes\omega}) 
                = \prod_{v \in S} \gamma_v^{\underline{\V}}(s,\pi\boxtimes\omega,\psi) 
                        L_S(1-s, \std\circ\phi_{\pi^\vee\boxtimes\omega^{-1}})
          \]
          holds, where
          \[
            L_S(s, \std\circ\phi_{\Pi\boxtimes\omega}) 
              = \prod_{v \not\in S}L_v(s, \std\circ\phi_{\Pi\boxtimes\omega})
          \]
          is the partial standard $L$-factor.
          \label{gfe}
\end{enumerate}
Moreover, the properties \eqref{ut},\eqref{ml}, \eqref{unra}, \eqref{split},\eqref{psi},\eqref{min},\eqref{GLn} and \eqref{gfe} determine $\gamma^\V(s,\pi\boxtimes\omega,\psi)$ uniquely.
\end{thm}

\subsection{$L$- and $\epsilon$-factors}

In this subsection we discuss about the $L$- and $\epsilon$-factors. 
We define the $L$- and $\epsilon$-factors as in \cite[\S10]{LR05}. Moreover, as in \cite{LR05}, we have
\begin{prop}
The $\epsilon$-factor $\epsilon^\V(s,\pi\boxtimes\omega, \psi)$ is a monomial of $q^{-s}$. 
\end{prop}
Note that Yamana gave another definition of the $L$-factor by g.c.d property \cite[Theorem 5.2]{Yam14}. Moreover, \cite[Lemma 7.2]{Yam14} implies that both $L$-factors coincide. Although he assume $\ch(F) = 0$ in \cite{Yam14}, the results above still hold in the case $\ch(F)\not=0,2$.


\subsection{Relation with other types of $\gamma$-factors}

In this subsection, we write the relation with other types of $\gamma$-factors. 
First, we compare $\gamma^\V(s,\pi\boxtimes\omega,\psi)$ to the Langlands-Shahidi $\gamma$-factor. 
Then, we compare $\gamma^\V(s,\pi\boxtimes\omega,\psi)$ to the $\gamma$-factor defined by Genestier-Lafforgue.

Before stating them, we define the doubling $\gamma$-factor of irreducible representations of $G(\V)^\circ\times{\rm Res}\GL_1$. Let $\pi$ be an irreducible representation of $G(\V)^\circ$, let $\omega$ be a character of $E^\times$, and let $\psi$ be a non-trivial character of $F$. Then we define the doubling $\gamma$-factor $\gamma^\V(s,\pi\times\omega,\psi)$ by $\gamma^\V(s, \tilde{\pi}\times\omega,\psi)$ where $\tilde{\pi}$ is an irreducible representation of $G(\V)$ so that the restriction of $\tilde{\pi}$ to $G(\V)^\circ$ contains $\pi$. By the discussion of \cite[\S3.4, \S3.6]{AG17}, we can prove that the definition of $\gamma^\V(s,\pi\times\omega,\psi)$ does not depend on the choice of $\tilde{\pi}$. 

Now we state the relation with the Langlands-Shahidi $\gamma$-factor. The theory of Langlands-Shahidi $\gamma$-factor is extended to a quasi-split connected reductive group over a function field by Lomeli \cite{Lom15a, Lom15b}. Suppose that $\V$ is of the type $(\SO), (\Sp), (\rU)$ and $G(\V)^\circ$ is quasi-split. For a generic irreducible representation of $G(\V)^\circ$, we denote by $\gamma^{LS}(\pi, r_1, \psi)$ the first Langlands-Shahidi $\gamma$-factor. Note that $r_1$ is associated with the standard homomorphism $\std$ defined in Sect.~\ref{reps} as follows: We define a representation $\std'$ of ${}^L\!(G(\V)^\circ\times\Res_{E/F}\GL_1)$ by
\[
\std'(g, z') = {}^t\!\std(g,1)^{-1} \cdot \std(1, z') 
\]
for $g \in \widehat{G(\V)^\circ}$ and $z' \in {}^L\!(\Res_{E/F}\GL_1)$. Then $r_1$ is equivalent to $\std'$ as representations of ${}^L\!(G(\V)^\circ\times\Res_{E/F}\GL_1)$.

\begin{cor}
Suppose that $\V$ is of the type $(\SO), (\Sp), (\rU)$ and $G(\V)^\circ$ is quasi-split.
Let $\pi$ be a generic irreducible representation of $G(\V)^\circ$, let $\omega$ be a character of $E^\times$, and let $\psi$ be a non-trivial additive character of $F$. Then we have
\[
\gamma^{LS}(\pi\boxtimes\omega_s, r_1, \psi) = \gamma^\V(s, \pi^\vee\boxtimes\omega, \psi).
\]
\end{cor}

\begin{proof}
To characterize the Langlands-Shahidi $\gamma$-factor $\gamma^{LS}(\pi\boxtimes\omega_s, r_1, \psi)$, it suffice to check the properties (i)-(vi) and (ix) of \cite[Theorems 5.1, 7.3]{Lom15a}. On the other hand, Theorem \ref{maingamma} implies that $\gamma^\V(s,\pi^\vee\boxtimes\omega,\psi)$ satisfies above properties.  
\end{proof}

Finally, we state the relation with the $\gamma$-factor defined by Genestier-Lafforgue. Put
\[
H = \begin{cases}
G(\V)^\circ \times \GL_1 & \mbox{in the cases $(\SO), (\Sp), (\rQ_1), (\rQ_{-1})$}, \\
G(\V) \times \Res_{E/F}(\GL_1) & \mbox{in the cases $(\rU), (\rqGL)$}, \\
G(\V) \times \GL_1\times \GL_1 & \mbox{in the cases $(\GL), (\rQGL)$}.
\end{cases}
\]
Fix an isomorphism $\iota:\C \rightarrow \overline{\Q}_l$ of field. By $\iota$, we can regard $\pi, \omega, \psi$ as an irreducible representation with coefficient in $\overline{\Q_l}$, a $\overline{\Q_l}^\times$ valued character, and a $\overline{\Q_l}^\times$ valued non-trivial additive character respectively. One can show that $\pi, \omega$ are defied over a finite extension field over $\Q_l$. Hence, according \cite[Th\'{e}or\`{e}me 8.1]{GL17}, we can associate a semisimple $L$-parameter $\sigma_{\pi\boxtimes\omega}$, that is, an $L$-parameter $\sigma_{\pi\boxtimes\omega}: W_F\times \SL_2(\overline{\Q_l}) \rightarrow {}^L\!H(\overline{\Q_l})$ such that $\sigma_{\pi\boxtimes\omega} |_{1\times\SL_2(\overline{\Q_l})}$ is trivial. 
Note that $\sigma_{\pi\boxtimes\omega}$ is expected to be the semisimplification of the $L$-parameter $\phi_{\pi\boxtimes\omega}$ of $\pi\boxtimes\omega$.
Thus, we can expect that $\gamma_F(s, \std\circ\sigma_{\pi\boxtimes\omega}, \psi)$ coincides with the standard $\gamma$-factor $\gamma_F(s, \phi_{\pi\boxtimes\omega}, \psi)$ where $\std$ is the standard homomorphism of ${}^L\!H(\overline{\Q_l})$ defined in Sect.\ref{std_loc}.

\begin{cor}
Assume that $\gamma_F(s, \std\circ\sigma_{1_\V\boxtimes1}, \psi) = \gamma^\V(s, 1_\V\boxtimes1, \psi)$ in the minimal cases. Then, we have
\[
\gamma_F(s, \std\circ\sigma_{\pi\boxtimes\omega}, \psi) = \gamma^\V(s, \pi\boxtimes\omega, \psi).
\]
\end{cor}


\section{
            Proof of the main theorem
            }

In this section, we explain the proof of Theorem \ref{maingamma}, which will be completed in Sect.~\ref{computation}.

\subsection{
                Uniqueness
                }

The uniqueness can be proved standard global argument (see. e.g. \cite{LR05}, \cite{Kak19}) with the globalization results \cite[Appendice I]{Hen84} and \cite[Theorem 1.1]{GL18}.

\subsection{
                Formal properties
                }

The properties \eqref{ut} unramified twisting, \eqref{ml} multiplicativity, \eqref{fe} functional equation, \eqref{sd} self duality, and \eqref{psi} dependence of $\psi$, are deduced from the framework of the doubling method.
One can prove them in the same line with \cite{LR05} (see also \cite[Remark 6.3]{Kak19}). 
By using the argument of the Eisenstein series (see for example \cite{Har74}), we have \eqref{gfe} global functional equation as in \cite{LR05}.

\subsection{
                Split cases
                }

\cite[\S6.3]{Kak19}.

\subsection{
                Unramified cases
                } 

\cite[\S7]{LR05}.

\subsection{
                Minimal cases
                }

This property is proved by explicit computation of the $\gamma$-factor and standard $\gamma$-factor respectively. The computation is done in Sect.~\ref{computation} below.

\subsection{
                $\GL_n$-factors
                }

\cite[Appendix]{Yam13}.


\section{
            The $\gamma$-factor of the trivial representation for the minimal cases
            }\label{computation}

In this section, we compute the standard $\gamma$-factor of the principal parameter and the $\gamma$-factor of the trivial representation in the minimal cases in the sense of Sect.~\ref{mincases} , and we complete the proof of Theorem \ref{maingamma}. 
We begin with the explanation of the standard $\gamma$-factor of the principal parameter in each case (Sect.~\ref{prfstd}). 
Then we give notations to compute the analytic $\gamma$-factor (Sect.~\ref{prfnote}). The computation consists of three parts: the cases $(\rSOa, n=2), (\rU, n=1), (\rQ_{-1}, n=1)$ (Sect.~\ref{n1unit}), the case $(\rQ_1, n=1)$ (Sect.~\ref{prf2}), and the cases $(\rSOa, n=3), (\rSOa, n=4), (\rUra, n=2)$ (Sect.~\ref{prf3}).

\subsection{
                 The standard $\gamma$-factor of the principal parameters
                }\label{prfstd}

In this subsection, we give the irreducible decomposition of the representation $\std\circ\phi_0$, and give the formula of standard $\gamma$-factor in the minimal cases. 

\begin{prop}
We denote by $r_m$ the unique $m$-dimensional irreducible representation of $\SL_2(\C)$. 
Then the irreducible decomposition of the representation $\std\circ\phi_0$ is given by the following: 
\begin{enumerate}
\item In the cases $(\rSOa, n=2), (\rQ_{-1}, n=1)$, 
        we have $\std\circ\phi_0 = (r_1 \boxtimes 1) \oplus (r_1 \boxtimes \chi_{\V}')$ 
         as representations of $\SL_2(\C)\times W_F$ 
        where $\chi_{\V}'$ is the character of $W_F$ 
        , which is associated with $\chi_\V$ via the local class field theory.
\item In the cases $(\rSOa, n=3), (\rQ_1,n=1)$, we have $\std\circ\phi_0 = r_N \boxtimes 1$ 
        as representations of $\SL_2(\C) \times W_F$.
\item In the cases $(\rSOa, n=4)$, we have $\std\circ\phi_0 = (r_3 \boxtimes 1) \oplus 
         (r_1 \boxtimes 1)$
        as representations of $\SL_2(\C) \times W_F$.
\item In the cases $(\rU, n=1), (\rUra, n=2)$, we have $\std\circ\phi_0 = (r_n \boxtimes 1)
        \oplus (r_2\boxtimes \chi_E)$ 
        as representations of $\SL_2(\C) \times W_E$.
\end{enumerate}
\end{prop}

Therefore, the $\gamma$-factor is calculated as follows:

\begin{prop}\label{stdgammatriv}
\begin{enumerate} 
\item In the cases $(\rSOa, n=2), (\rQ_{-1}, n=1)$, 
        \[
        \gamma_F(s+\frac{1}{2}, \std\circ\phi_0, \psi) = 
                   \frac{\zeta_F(-s+\frac{1}{2})}{\zeta_F(s+\frac{1}{2})}
                   \gamma_F(s+\frac{1}{2}, \chi_{\disc(\V)}, \psi),
        \]
\item In the case $(\rSOa, n=3)$, 
        \[
        \gamma_F(s+\frac{1}{2}, \std\circ\phi_0, \psi) =
                    \frac{\zeta_F(-s)}{\zeta_F(s)}\frac{\zeta_F(-s+1)}{\zeta_F(s+1)},
        \]
\item In the case $(\rSOa, n=4)$,
        \[
        \gamma_F(s+\frac{1}{2}, \std\circ\phi_0, \psi) = 
         \frac{\zeta_F(-s+\frac{3}{2})}{\zeta_F(s+\frac{3}{2})}
         \frac{\zeta_F(-s+\frac{1}{2})^2}{\zeta_F(s+\frac{1}{2})^2}
         \frac{\zeta_F(-s-\frac{1}{2})}{\zeta_F(s-\frac{1}{2})},
        \]
\item In the case $(\rQ_1, n=1)$,
        \[
        \gamma_F(s+\frac{1}{2}, \std\circ\phi_0, \psi) =
         \frac{\zeta_F(-s+\frac{3}{2})}{\zeta_F(s+\frac{3}{2})}
         \frac{\zeta_F(-s+\frac{1}{2})}{\zeta_F(s+\frac{1}{2})}
         \frac{\zeta_F(-s-\frac{1}{2})}{\zeta_F(s-\frac{1}{2})},
        \]
\item In the case $(\rU, n=1)$, 
         \[
         \gamma_E(s+\frac{1}{2}, \std\circ\phi_0, \psi) = 
                   \frac{\zeta_F(-s+\frac{1}{2})}{\zeta_F(s+\frac{1}{2})}
                   \gamma_F(s+\frac{1}{2}, \chi_E, \psi)
         \]
\item In the case $(\rUra, n=2)$,
        \[
        \gamma_E(s+\frac{1}{2}, \std\circ\phi_0, \psi) = 
        -q^{-s}\frac{\zeta_F(-s+1)}{\zeta_F(s+1)}\epsilon_F(s+\frac{1}{2}, \chi_E,\psi)^2.
        \]
\end{enumerate}
\end{prop}

\subsection{
                Preliminaries
                }\label{prfnote}

Let 
\[
v_1 = {}^t\!(1,0, \ldots, 0), \ldots, v_n = {}^t\!(0,\ldots,0,1)
\]
be a basis of $V = E^n$ over $E$. Then we fix the canonical basis of $V^\Box$ by $e_1, \ldots, e_{2n}$ defined by
\[
e_i = \begin{cases} (v_i, v_i) & 1 \leq i \leq n \\
               (v_{i-n}, -v_{i-n}) & n+1 \leq i \leq 2n \\
       \end{cases}
\]

We have a decomposition 
\[
G(\V^\Box) = P(V^\tru) (G(\V) \times G(\V))
\]
by \cite[Lemma 2.1]{GPSR87}. Thus, $I(s,1) \cong C(G(\V))$ as $G(\V) \times G(\V)$ modules, and $\Gamma^\V(s, 1_\V, 1, A, \psi)$ is the eigenvalue of $M^*(s, 1, A, \psi)$ on $C(G(\V))^{G(\V) \times G(\V)} = 1_G \cdot \C$. We compute them case by case.


\subsection{
                 The cases $(\rSOa, n=2), (\rU, n=1), (\rQ_{-1}, n=1)$
                }\label{n1unit}

Note that we assume $E$ is a division algebra. In each case, we may suppose that
\[
\begin{cases}
V = F^2, h = \langle \diag(1,-a) \rangle & \mbox{ in the case $(\rSOa, n=2)$}, \\
V =E, E = F(\alpha), h = \langle 1 \rangle & \mbox{ in the case $(\rU, n=1)$}, \\
V = E, E = (a, b /F), h = \langle \alpha \rangle  & \mbox{ in the case $(\rQ_{-1}, n=1)$} 
\end{cases}
\]
for some $a,b \in F^\times \backslash F^{\times 2}$, $\alpha \in E_0$ satisfying 
\begin{itemize}
\item $\ord_F(a) = 0, 1$ and $\ord_F(b) = 0,1$, 
\item $\ord_F(a) \not= \ord_F(b)$,
\item $\alpha^2 =a$.
\end{itemize}
We take $K := G(\V) \times G(\V) \subset G(\V^\Box)$, $w_1=(1, -1) \in K$, and $A \in \u(V^\tru)$ by 
\[
A=\begin{pmatrix} 0 & X^{-1} \\ 0 & 0 \end{pmatrix} 
\]
where 
\[
X=
\begin{cases}
  \begin{pmatrix} 0 & a \\ 1 & 0 \end{pmatrix} & \mbox{ in the case $(\rSOa, n=2)$}, \\
   \alpha & \mbox{ in the case $(\rU, n=1), (\rQ_{-1}, n=1)$}.
\end{cases}
\]

\begin{prop} Let $f$ be the $K$ invariant section of $I(0,1)$ with $f(1) = 1$. Then, we have 
\[
l_{\psi_A}(f_s) = \zeta_{F(\sqrt{a})}(s+ \frac{1}{2})^{-1}.
\] 
\end{prop}

\begin{proof}
At first, we note that 
\[
\begin{pmatrix} 1 & 0 \\ xX & 1 \\ \end{pmatrix}
= \begin{pmatrix} -x^{-1}X^{-1} & 1 \\ 0 & xX \end{pmatrix}
  \begin{pmatrix} 0 & 1 \\ 1 & x^{-1}X^{-1} \end{pmatrix}
\]
gives the Iwasawa decomposition in $G(\V^\Box)$ for $x \in F^\times$ if $\ord_F(x) < 0$. 

If $\ord_F(a) = 0$, 
\begin{align*}
\zeta_F(2s+1)f_s(\begin{pmatrix}1 & 0 \\ x X & 1 \end{pmatrix}) 
&= \begin{cases} \int_{F^\times} 1_{\cO}(t)|t|^{2s+1} \: d^\times t & \ord_F(x) \geq 0, \\
                       \int_{F^\times} 1_{t^{-1}\cO}(x)|y|^{2s+1} \: d^\times t & \ord_F(x) < 0
     \end{cases} \\
&=\int_{F^\times} 1_{\cO}(t) 1_{t^{-1}\cO}(x) |t|^{2s+1} \: d^\times t. 
\end{align*}
Thus, 
\begin{align*}
\zeta_F(2s+1)l_{\psi_A}(f_s) 
&=\int_{F^\times} 1_{\cO}(t) \left( \int_F 1_{t^{-1}\cO}(x) \psi(2x) \: dx \right) \: d^\times t\\
&=\int_{F^\times} 1_{\cO^\times}(t) \: d^\times t = 1.
\end{align*}
Since $\zeta_{F(\sqrt{a})}(s+ \frac{1}{2}) = \zeta_F(2s+1)$, we have the claim in this case.

If $\ord_F(a) = 1$, 
\begin{align*}
\zeta_F(2s+1)f_s(\begin{pmatrix}1 & 0 \\ xX & 1 \end{pmatrix}) 
&= \begin{cases} \int_{F^\times} 1_{\cO}(t)|t|^{s+\frac{1}{2}} \: d^\times t 
                       & \ord_F(x) \geq 0 \\
                     q^{s+\frac{1}{2}} \int_{F^\times} 1_{t^{-1}\cO}(x)|t|^{s+\frac{1}{2}} \: d^\times t 
                       & \ord_F(x) < 0
     \end{cases} \\
&=q^{s+\frac{1}{2}} \int_{F^\times} 1_{\cO}(t)1_{t^{-1}\cO}(x) |t|^{s+\frac{1}{2}} 
  \: d^\times t + 1_{\cO}(x) (1-q^{s+\frac{1}{2}})\zeta_F(2s+1) \\
&= q^{s+\frac{1}{2}}\int_{F^\times} 1_{\cO}(t)1_{t^{-1}\cO}(x) |t|^{s+\frac{1}{2}} 
  \: d^\times t - 1_{\cO}(x)\frac{q^{s+\frac{1}{2}}}{1+q^{-s-\frac{1}{2}}}.
\end{align*}
Thus, 
\begin{align*}
\zeta_F(2s+ 1)l_{\psi_A}(f_s)
&=q^{s+\frac{1}{2}}\int_{F^\times}1_{\cO^\times}(t) |t|^{s+\frac{1}{2}} \: d^\times t - \frac{q^{s+\frac{1}{2}}}{1+q^{-s-\frac{1}{2}}} \\
&=\frac{1}{1+ q^{-s-\frac{1}{2}}}.
\end{align*}
Since $\zeta_{F(\sqrt{a})}(s+\frac{1}{2}) = (1+q^{-s-\frac{1}{2}})\zeta_F(2s+1)$, we have the claim in this case.
\end{proof}

Thus we have:
\begin{prop}\label{dimu=1}
\begin{enumerate}
\item In the case $(\rSOa, n=2), (\rQ_{-1}, n=1)$, we have
        \[
        \gamma^\V(s + \frac{1}{2}, 1_\V\boxtimes1, \psi)
          = \frac{\zeta_{F(\sqrt{a})}(-s+\frac{1}{2})}{\zeta_{F(\sqrt{a})}(s+\frac{1}{2})}
             \epsilon_F(s+\frac{1}{2}, \chi_{\disc(\V)}, \psi).
        \]
\item In the case $(\rU, n=1)$, we have
        \[
        \gamma^\V(s + \frac{1}{2}, 1_\V\boxtimes1, \psi)
          = \frac{\zeta_E(-s+\frac{1}{2})}{\zeta_E(s+\frac{1}{2})}\epsilon_F(s+\frac{1}{2}, \chi_E, \psi).
        \]
\end{enumerate}
\end{prop}


\subsection{
                The case $(\rQ_1, n=1)$
                }\label{prf2}

Note that we assume that $E$ is a division quaternion algebra over $F$. In this case, $\dim_{F} \u(V^\tru) = 3$.

\begin{lem}\label{quat_tech}
Take $\alpha, \beta \in E_0$ so that $\ord_E(\alpha) =0,  \ord_E(\beta) = 1, \alpha\beta + \beta\alpha = 0$, and $E$ is spanned by $1,\alpha, \beta, \alpha\beta$ over $F$.
We denote $\alpha^2$ (resp. $\beta^2$) by $a$ (resp. $b$). Then, we have
\[
1_{t^{-1}\cO}(ax^2 + by^2 -abz^2) 
= 1_{t^{-1}\cO}(ax^2) 1_{t^{-1}\cO}(by^2) 1_{t^{-1}\cO}(abz^2)
\]
for $t \in F^\times$ and $x,y,z \in F$.
\end{lem}

\begin{proof}
At first, we note that $1,\alpha,\beta,\alpha\beta$ is a basis of $\cO_E$ over $\cO$.
If $\ord_F(t) = 2m$ for some $m \in \Z$, then 
\begin{align*}
&ax^2 + by^2 - abz^2 \in t^{-1}\cO \\
&\Leftrightarrow \alpha x + \beta y + \alpha\beta z \in \beta^{-2m}\cO_E \\
&\Leftrightarrow \alpha b^m x + \beta b^m y + \alpha\beta b^m z \in \cO_E \\
&\Leftrightarrow ax^2 \in t^{-1}\cO, by^2 \in t^{-1}\cO, abz^2 \in t^{-1}\cO.
\end{align*} 
If $\ord_F(t) = 2m+1$ for some $m \in \Z$, then
\begin{align*}
&ax^2 + by^2 - abz^2 \in t^{-1}\cO \\
&\Leftrightarrow \alpha x + \beta y + \alpha\beta z \in \beta^{-2m}\cO_E \\
&\Leftrightarrow -\alpha\beta b^mx + b^{m+1}y - \alpha b^{m+1} z \in \cO_E \\
&\Leftrightarrow ax^2 \in t^{-1}\cO, by^2 \in t^{-1}\cO, abz^2 \in t^{-1}\cO.
\end{align*}
Hence we have the claim.
\end{proof}

\begin{prop}
Take $\alpha,\beta \in E_0$ as in Lemma \ref{quat_tech}, a compact subgroup $K = G(\V) \times G(\V)$, a $K$ invariant section $f \in I(0, 1)$ with $f(1)=1$,  and 
\[
A = \begin{pmatrix} 0 & \alpha^{-1} \\ 0 & 0 \end{pmatrix}.
\]
Then, we have
\[
l_{\psi_A}(f_s) = \zeta_F(s + \frac{3}{2})^{-1} (1 + q^{-s + \frac{1}{2}}).
\]
\end{prop}

\begin{proof}
For $x,y,z \in F$, 
\begin{align*}
&\zeta_F(s+ \frac{3}{2})
f_s(\begin{pmatrix}1 & 0 \\ \alpha x + \beta y +\alpha\beta z & 1 \end{pmatrix}) \\
&= \begin{cases}
   \int_{F^\times} 1_{\cO}(t)|t|^{s+\frac{3}{2}} \: d^\times y 
          & \ord_F(ax^2 + by^2 - abz^2) \geq 0 \\
   \int_{F^\times} 1_{\cO}(t)|t|^{s+\frac{3}{2}} 
       |ax^2 + by^2 - abz^2|^{s + \frac{3}{2}} \: d^\times y
          & \ord_F(ax^2 + by^2 - abz^2) < 0
   \end{cases} \\
&= \int_{F^\times}1_{\cO}(t)1_{t^{-1}\cO}(ax^2 + by^2 - abz^2) |t|^{s+\frac{3}{2}} \: d^\times t
\end{align*}
By Lemma \ref{quat_tech}, this is equal to
\[
\int_{F^\times}1_{\cO}(t)1_{t^{-1}\cO}(ax^2) 1_{t^{-1}\cO}(by^2) 1_{t^{-1}\cO}(abz^2)
                      |t|^{s+\frac{3}{2}} \: d^\times t.
\]
Therefore, 
\begin{align*}
&\zeta_F(s+\frac{3}{2})l_{\psi_A}(f_s) \\
&=\int_{F^\times} 1_{\cO}(t)|t|^{s+\frac{3}{2}} 
    \left(\int_F 1_{t^{-1}\cO}(ax^2)\psi(2x) \: dx \right)
    \left(\int_{F} 1_{t^{-1}\cO}(by^2) \: dy \right)^2 \: d^\times t \\
&=\int_{F^\times} (1_{\cO^\times}(t) + 1_{\varpi_F \cO^\times}(t))  |t|^{s+\frac{3}{2}} 
    \left(\int_{F} 1_{t^{-1}\cO}(by^2) \: dy \right)^2 \: d^\times t  \\
& = 1 + q^{-s + \frac{1}{2}}.
\end{align*}
\end{proof}

We note that $\chi_{\disc(A)}$ is unramified, $\chi_{\disc(A)}(\varpi_F) = -1$, and 
\[
R(s,1,A,\psi) = \frac{1+q^{s-\frac{1}{2}}}{1+q^{-s-\frac{1}{2}}}.
\]
Thus we have:
\begin{prop}\label{dimu=3}
\[
\gamma^\V(s + \frac{1}{2} ,1 \boxtimes 1 ,\psi) 
   = \frac{\zeta_F(-s+\frac{3}{2})}{\zeta_F(s+\frac{3}{2})}
     \frac{\zeta_F(-s+\frac{1}{2})}{\zeta_F(s+\frac{1}{2})}
     \frac{\zeta_F(-s-\frac{1}{2})}{\zeta_F(s-\frac{1}{2})}. 
\]
\end{prop}


\subsection{
                The cases $(\rSOa, n=3), (\rSOa, n=4), (\rU, n=2), (\rQ_{-1}, n=2), (\rQ_{-1}, n=3)$
                }\label{prf3}
We may suppose that
       \[
       \begin{cases}
        V = F^3, h = \langle \diag(1,-a,-b)) \rangle & \mbox{in the case $(\rSOa, n=3)$}, \\
        V = F^4, h = \langle \diag(1,-a,-b,ab) \rangle & \mbox{in the case $(\rSOa, n=4)$}, \\
        E = F(\sqrt{b}), V =E^2, h = \langle \diag(1,-a) \rangle 
                                                                    & \mbox{in the case $(\rUra, n=2)$}.
        \end{cases}
        \]
        for some $a,b \in F^\times$ satisfying
        \[
        \ord_F(a) = 0, \  \ord_F(b) = 1, \  (a,b)_F = -1.
        \]
In each case, we associate an isotropic orthogonal (resp. hermitian) space $\V'$ to the space $\V$ as follows:
\[
\begin{cases}
\mbox{the orthogonal space $(F^3, \langle \diag(1,-1, -b) \rangle)$}
          & \mbox{in the case $(\rSOa, n=3)$}, \\
\mbox{the orthogonal space $(F^4, \langle\diag(1,-1,-b,b)\rangle)$}
          & \mbox{in the case $(\rSOa, n=4)$}, \\
\mbox{the hermitian space $(E^2, \langle\diag(1,-1)\rangle)$ over $E$}
          & \mbox{in the case $(\rUra, n=2)$}.
\end{cases}
\]
Put
\[
R := \begin{cases}
       \diag(1, a, 1) & \mbox{ in the case $(\rSOa, n=3)$} \\
       \diag(1, a, 1, a) & \mbox{ in the case $(\rSOa, n=4)$} \\
       \diag(1,a) & \mbox{ in the case $(\rUra, n=2)$}
      \end{cases}
\]
then we have an isometry
\[
\cD: V^\Box \rightarrow {V'}^\Box : v \mapsto \begin{pmatrix} R & 0 \\ 0 & 1 \end{pmatrix}v,
\]
and thus we have an isomorphism
\[
\cD: G(\V^\Box) \rightarrow G({\V'}^\Box): g \mapsto 
\begin{pmatrix} R & 0 \\ 0 & 1 \end{pmatrix} 
g \begin{pmatrix} R^{-1} & 0 \\ 0 & 1 \end{pmatrix}.
\]

\begin{lem}
In the case $(\rSOa, n=4), (\rUra, n=2)$, the following diagram is commutative;
\[
\xymatrix{
             I^{\V'}(s,1) \ar[rr]^-{M^{\V' *}(s,1, A' ,\psi)} \ar[d]_{\cD} 
              & & I^{\V'}(-s,1) \ar[d]^{\cD} \\
             I^\V(s,1) \ar[rr]^-{M^{\V *}(s,1,A,\psi)} & & I^\V(-s,1) 
             }
\]
where $A' \in \u({V'}^\tru)_n$ and $A = R^{-1}A' \in \u(V^\tru)_n$.
\end{lem}

Choose a compact subgroup $K' :=  G({\V'}^\Box) \cap \GL_{2n}(\cO_F)$ of $G({\V'}^\Box)$.  Then $\cD(G(\V)\times G(\V)) \subset K'$.
Note that we have $P({V'}^\tru)K' = G(V')$ since $P(V^\tru)(G(\V)\times G(\V))$ by \cite[Lemma 2.1]{GPSR87}.
Let $f'$ be the $K'$ invariant section of $I^{\V'}(s,1)$ with $f'(1) = 1$.
Note that $\cD (f'_s)$ is the $G(\V)\times G(\V)$ invariant section of $I^{\V}(s,1)$. 
Hence we have 
\begin{align}\label{annlarge}
  l_{\psi_A}(f_s') = \Gamma^{\V}(s,1_\V \boxtimes 1, A , \psi) l_{\psi_{A'}} (f_{-s}').
\end{align}
Similar statement hold for the case $(\rSOa, n=3)$.
We may rewrite $\eqref{annlarge}$ as 
\[
\gamma^{\V}(s + \frac{1}{2},1_\V \boxtimes 1, \psi) 
 = c \: \gamma^{\V'}(s+ \frac{1}{2}, 1_{\V'}\boxtimes 1, \psi) \frac{Z^{\V'}(f_s', \xi')}{Z^{\V'}(f_{-s}',\xi')}
\]
where $\xi'$ is the non-zero coefficient of the trivial representation $G(\V')$, and  
\[
c = \begin{cases} 1 & \mbox{in the case $(\rSOa, n=4)$}, \\
                                 -1 & \mbox{in the case $(\rSOa, n=3), (\rUra, n=2)$}. \end{cases}.
\]
Then, by using the global functional equation $\eqref{gfe}$ for the trivial representation of a special orthogonal (or a unitary) group over $F$ which is anisotropic precisely at two places, we find that 
\[
\left(\frac{Z^{\V'}(f_s', \xi')}{Z^{\V'}(f_{-s}',\xi')}\right)^2 = 1.
\]
Therefore, the meromorphic function
\[
\frac{Z^{\V'}(f_s', \xi')}{Z^{\V'}(f_{-s}',\xi')} 
\]
is a constant function whose value is $1$ or $-1$. 
To determine the signature, we focus on the behavior of $Z^{\V'}(f_s',\xi')$ at $s=0$. 
\begin{lem}
Put
\[
\Omega(s) =\begin{cases}
                 Z^{\V'}(f_s', \xi')\zeta_F(s)^{-1} & \mbox{in the cases $(\rSOa, n=3), (\rUra, n=2)$}, \\
                 Z^{\V'}(f_s', \xi') & \mbox{in the case $(\rSOa, n=4)$}.
                 \end{cases} 
\]
Then, the meromorphic function $\Omega(s)$ has neither a pole nor a zero at $s=0$.
\end{lem}

\begin{proof}
Put 
\[
v_1' = ^t\!(1,0,\ldots,0), \ \ldots, \ v_n' = ^t\!(0, \ldots, 0, 1)
\]
a basis of $V'$ over $E$, and put $W'$ the subspace of $V'$ spanned by the isotropic vector $v_1' + v_2'$. Then there is the exact sequence
\[
1 \rightarrow U({W'}^\Box) \rightarrow P({W'}^{\tru}) \rightarrow \GL_2(E) \times G(\W_0^\Box) \rightarrow 1. 
\]
We denote by $K'' \times K'''$ the image of $P({W'}^{\Box}) \cap K'$ in $\GL_2(E) \times G(\W_0^\Box)$. 

Now consider the trivial representation of $G'$. It is a subrepresentation of $\Ind_{P(W')}^{G(\V')}\delta_{P(W')}^{-1/2}$. Note that $\delta_{P(W')}^{-1/2} = | \ |^{\lambda} \otimes 1: \GL_1(E) \times G(\W_0)\rightarrow \C^\times$ where
\[
\lambda = \begin{cases} \frac{1}{2} & \mbox{in the cases $(\rSOa, n=3), (\rUra, n=2)$}, \\
                                  1 & \mbox{in the case $(\rSOa, n=4)$}. \end{cases}
\]
Moreover, we have
\[
Z^{\V'}(f_s', \xi') = J(s) \cdot Z^{\W_1'}(f_s'', \xi'')Z^{\W_0'}(f_s''', \xi''')
\]
where $f_s''$ (resp. by $f_s'''$) is the unique non-zero section in $I^{\W_1'}(s,1)^{K''}$ (resp. in $I^{\W_0'}(s,1)^{K'''}$), $\xi''$ (resp. $\xi'''$) is a non-zero coefficient of the representation $| \ |^\lambda$ of $\GL_1(E)$ (resp. the representation $1$ of $G(\W_0')$), and 
\[
J(s) = \int_{U({W'}^\Box) \cap P({V'}^\tru) \backslash U({W'}^\Box)} f_s'(u) \: du.
\] 
The integral defining $J(s)$ converges in $\Re s > -\frac{1}{2}$ (\cite[Lemma 5.1]{Yam14}), and the integrand $f_0'(u)$ is positive for all $u \in U({W'}^\Box)$. 
Hence $J(s)$ has neither pole nor zero at $s=0$.
On the other hand, we have
\[
Z^{\W_1'}(f_s'', \xi'')Z^{\W_0'}(f_s''',\xi''') =
\begin{cases}
\frac{\zeta_E(s)\zeta_E(s+1)}{\zeta_E(2s+1)} & \mbox{in the case $(\rSOa, n=3), (\rUra, n=2)$},\\ 
\frac{\zeta_F(s-\frac{1}{2})\zeta_F(s+\frac{1}{2})^2\zeta_F(s+\frac{3}{2})}
       {\zeta_F(2s+1)\zeta_F(2s+3)} & \mbox{in the case $(\rSOa, n=4)$}
\end{cases}
\]
by \cite[Proposition 3]{LR05}. Therefore we can conclude that $\Omega(s)$ has neither pole nor zero at $s=0$. Hence we have the lemma. 
\end{proof}

From this lemma, we have 
\[
\left.\frac{Z^{\V'}(f_s', \xi')}{Z^{\V'}(f_{-s}',\xi')} \right|_{s=0}  
 = \begin{cases}
    -1& (\rSOa, n=3), (\rUra, n=2), \\
    1 & (\rSOa,n=4).
    \end{cases}
\]

Summarizing, we have: 
\begin{prop}\label{highercases}
\begin{enumerate}
\item In the case $(\rSOa, n=3)$,
        \[
         \gamma^{\V}(s+\frac{1}{2}, 1_\V\boxtimes 1, \psi)
            = \frac{\zeta_F(-s)}{\zeta_F(s)} \frac{\zeta_F(-s+1)}{\zeta_F(s+1)}.
        \]
\item In the case $(\rSOa, n=4)$,
        \[
         \gamma^{\V}(s+\frac{1}{2}, 1_\V\boxtimes 1, \psi) 
          =\frac{\zeta_F(-s+\frac{3}{2})}{\zeta_F(s+\frac{3}{2})}
           \frac{\zeta_F(-s+\frac{1}{2})}{\zeta_F(s+\frac{1}{2})}
           \frac{\zeta_F(-s-\frac{1}{2})}{\zeta_F(s-\frac{1}{2})}.
        \]
\item In the case $(\rUra, n=2)$,
        \[
         \gamma^{\V}(s+\frac{1}{2}, 1_\V\boxtimes 1, \psi) 
         = \gamma_F(s, 1, \psi) \gamma_F(s, \chi_E, \psi)
             \gamma_F(s+1, 1, \psi)\gamma_F(s+1, \chi_E, \psi).
        \]
\end{enumerate}
\end{prop}

Finally, comparing Propositions \ref{dimu=1}, \ref{dimu=3}, \ref{highercases} to Proposition \ref{stdgammatriv}, we have Theorem \ref{maingamma} \eqref{min} and we complete the proof of Theorem \ref{maingamma}.



\end{document}